\pdfoutput=1
\RequirePackage{ifpdf}
\ifpdf 
\documentclass[pdftex]{sigma}
\else
\documentclass{sigma}
\fi

\usepackage{tikz-cd}

\numberwithin{equation}{section}

\newtheorem{Theorem}{Theorem}[section]
\newtheorem*{Theorem*}{Theorem}
\newtheorem{Corollary}[Theorem]{Corollary}
\newtheorem{Lemma}[Theorem]{Lemma}
\newtheorem{Proposition}[Theorem]{Proposition}
 { \theoremstyle{definition}
\newtheorem{Definition}[Theorem]{Definition}

\newtheorem{Example}[Theorem]{Example}
\newtheorem{Notations}[Theorem]{Notations}
\newtheorem{Remark}[Theorem]{Remark} }

\begin{document}
\allowdisplaybreaks

\newcommand{\arXivNumber}{2011.07527}

\renewcommand{\PaperNumber}{043}

\FirstPageHeading

\ShortArticleName{Difference Equation for Quintic 3-Fold}

\ArticleName{Difference Equation for Quintic 3-Fold}

\Author{Yaoxinog WEN}

\AuthorNameForHeading{Y.~Wen}

\Address{Korea Institute for Advanced Study, Seoul, 02455, Republic of Korea}
\Email{\href{mailto:y.x.wen.math@gmail.com}{y.x.wen.math@gmail.com}}

\ArticleDates{Received September 28, 2021, in final form June 04, 2022; Published online June 14, 2022}

\Abstract{In this paper, we use the Mellin--Barnes--Watson method to relate solutions of a~certain type of $q$-difference equations at $Q=0$ and $Q=\infty$. We consider two special cases; the first is the $q$-difference equation of $K$-theoretic $I$-function of the quintic, which is degree~25; we use Adams' method to find the extra 20 solutions at $Q=0$. The second special case is a fuchsian case, which is confluent to the differential equation of the cohomological $I$-function of the quintic. We compute the connection matrix and study the confluence of the $q$-difference structure.}


\Classification{14N35; 33D90; 39A13}

\section{Introduction}
Since the 1990s, the development of mirror symmetry has changed how people work on enumerative geometry and has made some surprising predictions in algebraic geometry. Calabi--Yau manifolds are essential in mirror symmetry. Among them, the quintic threefold was the first example for which mirror symmetry was used to make enumerative predictions~\cite{CdGP}.

There are several ways to state mirror symmetry. In Givental's approach to mirror symmetry, two cohomology valued formal functions play a crucial role, i.e., the so-called $J$-function and $I$-function. The $J$-function, by definition, encodes all the genus zero Gromov--Witten invariants, so it is essential. However, it is pretty hard to obtain an explicit formula. On the other hand, the $I$-function given by the oscillatory integral is computable. In~\cite{givental1998} Givental proved that $I$-function lies on the range of big $J$-function, and up to a change of coordinate, we can obtain $J$-function from $I$-function.

Let $X$ be the Fermat quintic, considered as a degree 5 hypersurface in $\mathbb{P}^4$, the cohomological $I$-function of $X$ is as follows
\begin{gather*}
I^{\rm coh}_{X}(\hbar,{\rm e}^t) = \sum_{d=0}^{\infty} \frac{\prod_{k=1}^{5d}(5H+k\hbar)}{\prod_{k=1}^{d}(H+k\hbar)^5} {\rm e}^{t\left(\frac{H}{\hbar}+d\right)}, 	
\end{gather*}
where $H$ is the hyperplane class of $\mathbb{P}^4$, and $\hbar$ is the equivariant parameter, since $H^4=0$ in the cohomology of $X$, the $I$-function of quintic satisfies the following degree 4 differential equation which is called the Picard--Fuchs equation
\begin{gather*}
\bigg[ \bigg(\hbar\frac{\rm d}{\rm dt}\bigg)^4\!\!-5^5 {\rm e}^t \bigg(\hbar\frac{\rm d}{\rm dt}+\frac{1}{5}\hbar\bigg)\bigg(\hbar\frac{\rm d}{\rm dt}+\frac{2}{5}\hbar\bigg) \bigg(\hbar\frac{\rm d}{\rm dt}+\frac{3}{5}\hbar\bigg)\bigg(\hbar\frac{\rm d}{\rm dt}+\frac{4}{5}\hbar\bigg) \bigg] I^{\rm coh}_X\big(\hbar,{\rm e}^t\big) = 0.	
\end{gather*}

Let $Q=5^5{\rm e}^t$, the above differential equation becomes the following form
\begin{gather}
\bigg[ \bigg(Q\frac{\rm d}{{\rm d}Q}\bigg)^4\!\!- Q\bigg(Q\frac{\rm d}{{\rm d}Q}+\frac{1}{5}\bigg)\bigg(Q\frac{\rm d}{{\rm d}Q}+\frac{2}{5}\bigg) \bigg(Q\frac{\rm d}{{\rm d}Q}+\frac{3}{5}\bigg)\bigg(Q\frac{\rm d}{{\rm d}Q}+\frac{4}{5}\bigg) \bigg] I^{\rm coh}_X(\hbar,Q) = 0.	
\label{differential-equation}
\end{gather}
The fundamental solutions at $Q=0$ are given by the expansion of $I^{\rm coh}_X(\hbar,Q)$, i.e.,
\begin{gather*}
I^{\rm coh}_X(\hbar,Q)= I_0 \cdot 1 + I_1 \cdot H + I_2 \cdot H^2 + I_3 \cdot H^3 \mod\big(H^4\big).
\end{gather*}
More preciously, the coefficients of the $I$-function relative to the cohomology basis give the fundamental solutions of (\ref{differential-equation}). Moreover, the fundamental solutions at $Q=\infty$ are related to FJRW theory \cite{CR13} and could be constructed explicitly via the Frobenius method.

Around 2000, Givental \cite{givental2000wdvv} and Lee \cite{lee2004quantum} introduced the $K$-theoretic Gromov–Witten (GW) invariants, these invariants are defined by replacing cohomological definitions by their $K$-theore\-ti\-cal analogs. The $K$-theoretic $J$-function and $I$-function are also defined and studied; unlike cohomological GW theory, the $K$-theoretic $I$-function satisfies a difference equation instead of a~differential equation. For example, let us still consider quintic $X$, denote by $q^{Q\partial_Q}$ the difference operator shifting $Q^k$ by $q^kQ^k$, the $K$-theoretic $I$-function of $X$ is as follows
\begin{gather}
I^K_X(q,Q)= P^{l_q(Q)}\sum_{d=0}^{\infty} \frac{\prod_{k=1}^{5d}\big(1-P^5q^k\big)}{\prod_{k=1}^{d}\big(1-Pq^k\big)^5} Q^d, \label{K-I-function}
\end{gather}
where $P=\mathcal{O}(-1)$ on $\mathbb{P}^4$, and
\begin{gather*}
l_q(Q) = -Q\frac{\theta_q^{\prime}(Q)}{\theta_q(Q)}
\end{gather*}
is the $q$-logarithm function. Here $\theta_q(Q)$ is the Jacobi's theta function and the $q$-logarithm function satisfies
\begin{gather*}
q^{Q\partial_Q}(l_q(Q)) = l_q(Q)+1. 	
\end{gather*}
Since $(1-P)^5=0$ in $K\big(\mathbb{P}^4\big)$, the $K$-theoretic $I$-function of $X$ satisfies the following difference equation
\begin{gather}
\Bigg[	\big(1-q^{Q\partial_Q}\big)^5 - Q \prod_{k=1}^5 \big(1-q^{5Q\partial_Q+k}\big) \Bigg]I^K_X(q,Q)=0. \label{difference-equation}
\end{gather}
It is a degree 25 difference equation but not fuchsian (Definition \ref{def-fuchsian}).

The characteristic equation (see (\ref{char-equ}) for definition) at $Q=\infty$ is
\begin{gather*}
\big(1-q^{-1}x^5\big)\big(1-q^{-2}x^5\big)\big(1-q^{-3}x^5\big)\big(1-q^{-4}x^5\big)\big(1-q^{-5}x^5\big)=0,	
\end{gather*}
with 25 distinct roots. Using Frobenius method, we obtain 25 solutions (at $Q=\infty$) of (\ref{difference-equation}) given by
\begin{gather*}
	W_{l,m}(1/Q) = e_{q,q^{\frac{l}{5}}\xi^m}(1/Q) \sum_{d\geq0}\frac{\prod_{k=0}^{d-1}\big(1-\xi^{-m}q^{-k-\frac{l}{5}}\big)^5}{\prod_{k=0}^{5d-1}\big(1-q^{-k-l}\big)}Q^{-d},
\end{gather*}
where
\begin{gather*}
\xi^5=1, \qquad l=1,\dots,5, \qquad	m=0,\dots,4,
\end{gather*}
and
\begin{gather*}
e_{q, \lambda_{q}}(Q)=\frac{\theta_{q}(Q)}{\theta_{q}(\lambda_{q} Q)} \in \mathcal{M}(\mathbb{C}^{*}).	
\end{gather*}
The function $e_{q,\lambda_q}$ satisfies the	 $q$-difference equation $q^{Q \partial_{Q}} e_{q, \lambda_{q}}(Q)=\lambda_{q} e_{q, \lambda_{q}}(Q)$.

The characteristic equation at $Q=0$ is
\begin{gather*}
(1-x)^5=0,	
\end{gather*}
with 5 multiple roots. One may obtain explicit formulas for solutions via the Frobenius method. However, we could obtain 5 solutions at $Q=0$ from $I^K_X(q, Q)$ as mentioned above. We use Adams' method to obtain the rest solutions.

\begin{Proposition}
The exceptional $20$ solutions of \eqref{difference-equation} are given as follows: for each $20$th root of unity $\xi$, we have a solution of form
\begin{gather*}
	e_{p,\xi p^{-1/2}}(z)F(z) = e_{p,\xi p^{-1/2}}(z)\sum_{n \geq 0}f_nz^n,
\end{gather*}
where $z=Q^{\frac{1}{20}}$, $p=q^{\frac{1}{20}}$. Let $\sigma_p=p^{z\partial_z}$, then $F(z)$ satisfies
\begin{gather*}
\Big[\big(z-\xi p^{-\frac{9}{2}}\sigma_p\big)\big(z-\xi p^{-\frac{7}{2}}\sigma_p\big) \big(z-\xi p^{-\frac{5}{2}}\sigma_p\big)\big(z-\xi p^{-\frac{3}{2}}\sigma_p\big)\big(z-\xi p^{-\frac{1}{2}} \sigma_p\big)
\\ \qquad
{}-\! \big(z^5\!-\!\xi^5p^{-\frac{25}{2}}\sigma_p^5\big)\big(z^5\!-\!\xi^5 p^{-\frac{15}{2}}\sigma_p^5\big)\big(z^5\!-\!\xi^5p^{-\frac{5}{2}}\sigma_p^5\big) \! \big(z^5\!-\!\xi^5p^{\frac{5}{2}}\sigma_p^5\big)\big(z^5\!-\!\xi^5p^{\frac{15}{2}}\sigma_p^5\big) \Big] F(z)\!=\!0.\!
\end{gather*}
\end{Proposition}

To the best of the author's knowledge, the connection matrix is critical in classifying fuchsian difference equations, and very little is known about the non-fuchsian case. Even in the fuchsian case, the connection matrix is hard to obtain if the characteristic equation has multiple roots.

We construct the following $K\big(\mathbb{P}^{n-1}\big)$ valued $q$-series motived by \cite{CR13} to obtain connection matrix
\begin{gather}
F_{m,n}(Q)=P^{l_q(Q)}\sum_{d=0}^{\infty} \frac{\prod_{i=1}^{m}(P{\alpha_i};q)_d}{(Pq;q)_d^n} Q^d. \label{aux-F_{m,n}}
\end{gather}
Here we use $q$-Pochhammer symbol notation:
\begin{gather*}
(a;q)_d : = (1-a)(1-qa)\cdots\big(1-q^{d-1}a\big) \qquad \text{for} \quad d>0.
\end{gather*}
Since $(1-P)^n=0$ in $K\big(\mathbb{P}^{n-1}\big)$, then (\ref{aux-F_{m,n}}) satisfies the following difference equation
\begin{gather}
	\Bigg[\big(1-q^{Q\partial_Q}\big)^n-Q\prod_{i=1}^{m}\big(1-{\alpha_i}q^{Q\partial_Q}\big) \Bigg] F_{m,n}(Q)=0 \mod\big((1-P)^n\big). \label{dif-F_{m,n}}
\end{gather}
Suppose $\alpha_i \notin \alpha_j q^{\mathbb{Z}\backslash \{0\}}$ then we could find the explicit formula for $m$ solutions at $Q=\infty$ denoted by $\{ W_k(1/Q) \}_{k=1}^m$ and we use Mellin--Barnes--Watson method to related solutions at~$Q=0$ and~$Q=\infty$.

\begin{Theorem}\label{Theorem1.1}
For $m \geq n$, the $K\big(\mathbb{P}^{n-1}\big)$ valued $q$-series has the following analytic continuation:
\begin{gather*}
P^{l_q(Q)}\sum_{d=0}^{\infty} \frac{\prod_{i=1}^{m}(P{\alpha_i};q)_d}{(Pq;q)_d^n} Q^d
= P^{l_q(Q)} \frac{\prod_{i=1}^m(P\alpha_i;q)_\infty}{(Pq;q)_{\infty}^n}\sum_{j=1}^m \frac{(q,q,P\alpha_j Q,q/(P\alpha_j Q);q)_\infty}{(P\alpha_j, q/(P\alpha_j),Q,q/Q;q )_{\infty}}
\\ \hphantom{P^{l_q(Q)}\sum_{d=0}^{\infty} \frac{\prod_{i=1}^{m}(P{\alpha_i};q)_d}{(Pq;q)_d^n} Q^d=}
{}\times \frac{\big(\alpha_j^{-1}q;q\big)^n_\infty {\rm e}^{-1}_{q,\alpha_j}(1/Q)}{\prod_{i=1,i\neq j}^m(\alpha_i/\alpha_j;q)_\infty (q;q)_\infty} W_j(1/Q).
\end{gather*}
\end{Theorem}

As for applications, if we take $n=5$, $m=25$ and $\{ \alpha_i \}_{i=1}^{25}=\big\{ \xi^l q^{\frac{k}{5}} \mid k,l=1,2,3,4,5 \big\}$, then~(\ref{aux-F_{m,n}}) becomes (\ref{K-I-function}). And if we take $n=4$, $m=4$ and $\{ \alpha_i\}_{i=1}^4=\big\{q^{\frac{i}{5}} \big\}_{i=1}^4$, then (\ref{dif-F_{m,n}}) becomes
\begin{gather}
	\Bigg[ \big(1-q^{Q\partial_Q}\big)^4-Q\prod_{i=1}^{4}\big(1-q^{\frac{i}{5}}q^{Q\partial_Q}\big) \Bigg] F(Q) = 0. \label{intro-difference}
\end{gather}
This difference equation is a lift of the differential equation (\ref{differential-equation}), i.e., if we let $q \rightarrow 1$, then~(\ref{intro-difference}) becomes~(\ref{differential-equation}). This phenomenon is called \emph{confluence} which was studied first by J.~Sauloy in~2000~\cite{Sauloy00}. Under the above specific choice, the formula in Theorem~\ref{Theorem1.1} becomes
\begin{gather*}
 P^{l_q(Q)}\sum_{d=0}^{\infty} \frac{\prod_{i=1}^{4}\big(Pq^{\frac{i}{5}};q\big)_d}{(Pq;q)_d^4} Q^d
= P^{l_q(Q)} \frac{\prod_{i=1}^4\big(Pq^{\frac{i}{5}};q\big)_\infty}{(Pq;q)_{\infty}^4} \sum_{j=1}^4 \frac{\big(q,q,Pq^{\frac{j}{5}} Q,q/\big(Pq^{\frac{j}{5}} Q\big);q\big)_\infty}{\big(Pq^{\frac{j}{5}}, q/\big(Pq^{\frac{j}{5}}\big),Q,q/Q;q \big)_{\infty}}
\\ \hphantom{P^{l_q(Q)}\sum_{d=0}^{\infty}\frac{\prod_{i=1}^{4}\big(Pq^{\frac{i}{5}};q\big)_d}{(Pq;q)_d^4} Q^d=}
{}\times\frac{\big(q^{-\frac{j}{5}}q;q\big)^4_\infty {\rm e}^{-1}_{q,q^{j/5}}(1/Q)}{\prod_{i=1,i\neq j}^4\big(q^{\frac{i-j}{5}};q\big)_\infty (q;q)_\infty } W_j(1/Q),
\end{gather*}
where $\{ W_j(1/Q) \}_{j=1}^4 $ are the fundamental solutions at $Q=\infty$. If we expand two sides with respect to $K$-group basis $(1-P)^k$, $k=0,1,2,3$, we obtain the connection matrix, for more details, see Section~\ref{section5.1}. Besides, the fundamental solutions of (\ref{intro-difference}) at $Q=0$ and $Q=\infty$ are confluent to the solutions of~(\ref{differential-equation}), finally, we compute the confluence of the connection matrix.

The paper is arranged as follows. Section~\ref{sec2} reviews some basic definitions and concepts of difference equations and introduces some special functions. In Section~\ref{sec3}, we use the difference equation of quintic as an example, and we use Adams' method and Frobenius method to solve the degree 25 difference equation at $Q=0$ and $Q=\infty$ respectively. In Section~\ref{sec4}, we generalize the difference equation for the quintic and construct a $K$-group valued series, and then we use the Mellin--Barnes--Watson method to relate solutions at $Q=0$ and $Q=\infty$. In Section~\ref{sec5}, we apply the results in Section~\ref{sec4} to a particular fuchsian case, and we expand the formula with respect to the $K$-group basis $(1-P)^k$ to find the connection matrix. Since the particular fuchsian case is confluent to the differential equation of quintic. In Section~\ref{sec6}, we study the confluence of the connection matrix.

\section{Preliminaries}\label{sec2}
In this section, we define some basic notions in the theory of $q$-difference equations. The main references are \cite{Ro19, Sauloy00,hardouin:hal-01959879}.

\begin{Notations} Here are some standard notations of general use:
\begin{itemize}\itemsep=0pt
\item[--] $Q$ and $q$ are complex variables and $|q|<1$, $q \neq 0$,
\item[--] $\mathbb{C}(\{Q \})$ is the field of meromorphic germs at 0, is the quotient field of $\mathbb{C} \{ Q \}$,
\item[--] $\mathcal{M}(\mathbb{C}^*)$ is the field of meromorphic functions on $\mathbb{C}^*$,
\item[--] $\mathcal{M}(\mathbb{C}^*,0)$ is the ring of germs at punctured neighborhood of $Q=0$,
\item[--] $\mathcal{M}\left(\mathbb{E}_{q}\right)$ is the field of meromorphic functions on elliptic curve $\mathbb{E}_q=\mathbb{C}^*/q^\mathbb{Z}$, i.e, the field of elliptic functions.
\item[--] $(a;q)_d=(1-a)(1-qa)\cdots\big(1-q^{d-1}a\big)$ for $d \in \mathbb{N} \cup \{+\infty \} $ is the $q$-Pochhammer symbol.
\end{itemize}
\end{Notations}

\begin{Definition}
A difference field is a pair $(K,{\sigma})$, where $K$ is a field, and $\sigma$ is a field automorphism of $K$.	
\end{Definition}

\begin{Example}
We will focus on the fields in the above notations,
\begin{gather*}
 \mathcal{M}(\mathbb{C}^*) \subset \mathcal{M}(\mathbb{C}^*,0),
\end{gather*}
they are all endowed with the $q$-shift operator $\sigma_q := q^{Q \partial_{Q}}\colon f(Q) \mapsto f(qQ)$. Let $K=\mathcal{M}(\mathbb{C}^*)$ or~$\mathcal{M}(\mathbb{C}^*,0)$. Usually, we denote the field of constants of the difference field $(K, \sigma_q)$ as $K^{\sigma_q}$. For example, $\mathcal{M}(\mathbb{C}^*)^{\sigma_q}=\mathcal{M}(\mathbb{C}^*)^{\sigma_q} =\mathcal{M}(\mathbb{E}_{q})$. This is the main reason that the modular form such as elliptic function appears naturally in the theory of $q$-difference equation.
\end{Example}

\subsection[Regular singular q-difference equations]
{Regular singular $\boldsymbol q$-difference equations}
\begin{Definition}
Let $(E_{q})\colon q^{Q \partial_{Q}} X_{q}(Q)=A_{q}(Q) X_{q}(Q)$ be a $q$-difference system, with $A_{q} \in {\rm GL}_{n}(K)$. We define the solution space of this $q$-difference equation by
\begin{gather*}
\operatorname{Sol}(E_{q})=\big\{X_{q} \in K^{n} \mid q^{Q \partial_{Q}} X_{q}(Q)=A_{q}(Q) X_{q}(Q)\big\}.
\end{gather*}	
\end{Definition}

\begin{Remark}
From now on, we will focus on the local solutions at $Q=0$, and the results will also hold for $Q=\infty$. The reason why we don't consider solutions at other singular points is that: if a function $f(Q)$ is a solution of a~$q$-difference equation $q^{Q\partial_Q}f(Q)=a(Q)f(Q)$ and has a singularity at some $Q_0 \neq 0,\infty$, then $f(Q)$ has a singularity at any complex number~$Q_0q^k$.	
\end{Remark}

\begin{Proposition}[{\cite[Theorem 2.3.1, p.~118]{hardouin:hal-01959879}}]
Let $(E_{q})\colon q^{Q \partial_{Q}} X_{q}(Q)=A_{q}(Q) X_{q}(Q)$ be a~$q$-dif\-ference system. Then, we have
\begin{gather*}
\operatorname{dim}_{\mathcal{M}(\mathbb{E}_{q})}\big(\operatorname{Sol}(E_{q})\big) \leq \operatorname{rank}(A_{q}).
\end{gather*}
\end{Proposition}

\begin{Definition}
Let $q^{Q \partial_{Q}} X_{q}(Q)=A_{q}(Q) X_{q}(Q)$ be a $q$-difference system. A fundamental solu\-tion of this system is an invertible matricial solution $\mathcal{X}_{q} \in {\rm GL}_{n}(K)$ such that $q^{Q \partial_{Q}} \mathcal{X}_{q}(Q)=A_{q}(Q) \mathcal{X}_{q}(Q)$.	
\end{Definition}

\begin{Definition}
Let $q^{Q \partial_{Q}} X_{q}(Q)=A_{q}(Q) X_{q}(Q)$ be a $q$-difference system. Consider a matrix $P_{q} \in {\rm GL}_{n}(K)$. The gauge transform of the matrix $A_{q}$ by the gauge transformation $P_q$ is the matrix
\begin{gather*}
	P_{q} \cdot[A_{q}]:=\big(q^{Q \partial_{Q}} P_{q}\big) A_{q} P_{q}^{-1}.
\end{gather*}
A second $q$-difference system $q^{Q \partial_{Q}} X_{q}(Q)=B_{q}(Q) X_{q}(Q)$ is said to be equivalent (over $K$) by gauge transform to the first one if there exists a matrix $P_{q} \in {\rm GL}_{n}(K)$ such that
\begin{gather*}
B_{q}=P_{q} \cdot[A_{q}].	
\end{gather*}
\end{Definition}

Let us define the regular singular $q$-difference equation. We shall start from the local analytic study, i.e., taking field $\mathbb{C}(\{ Q \})$, and then look for solutions in the field $K=\mathcal{M}(\mathbb{C}^*)$ or $\mathcal{M}(\mathbb{C}^*,0)$.

\begin{Definition}
Let $A_{q} \in {\rm GL}_{n}(\mathbb{C}(\{ Q \}))$, a system $q^{Q \partial_{Q}} X_{q}(Q)=A_{q}(Q) X_{q}(Q)$ is said to be regular singular at $Q=0$ if there exists a $q$-gauge transform $P_{q} \in {\rm GL}_{n}(\mathbb{C}(\{Q\}))$ such that the matrix $(P_{q} \cdot[A_{q}])(0)$ is well-defined and invertible: $P_{q} \cdot[A_{q}](0) \in {\rm GL}_{n}(\mathbb{C})$.	
\end{Definition}

\begin{Definition}
Consider a regular singular $q$-difference system $q^{Q\partial_Q}X_q(Q) = A_q(Q)X_q(Q) $. Suppose $A_q(0) \in {\rm GL}_n(\mathbb{C})$ and denote by $(\lambda_i)$ the eigenvalues of the matrix $A_q(0)$. This $q$-difference system is said to be non $q$-resonant if for every $i \neq j$, we have $\frac{\lambda_i}{\lambda_j} \notin q^{\mathbb{Z} \backslash \{0\} }$, where $	q^{\mathbb{Z}\backslash \{0\}}:=\big\{q^{k} \mid k \in \mathbb{Z}\backslash \{0\} \big\} \subset \mathbb{C} $.	
\end{Definition}

Let's introduce some special functions which are needed to solve regular singular $q$-difference equations.

We define Jacobi's theta function by
\begin{gather*}
\theta_{q}(Q) =\sum_{d \in \mathbb{Z}} q^{\frac{d(d-1)}{2}} Q^{d}.
\end{gather*}
This function satisfies the $q$-difference equation $q^{Q \partial_{Q}} \theta_{q}(Q)=\frac{1}{Q} \theta_{q}(Q)$. And it has a famous Jacobi's triple identity
\begin{gather*}
	\theta_q(Q)= (q;q)_\infty (-Q;q)_\infty (-q/Q;q)_\infty.
\end{gather*}
In the following, we define two special functions which are essential in solving regular singular (irregular) $q$-difference equations.

\begin{Definition}
Let $\lambda_{q} \in \mathbb{C}^{*}$. The $q$-character associated to $\lambda$ is the function $e_{q, \lambda_{q}} \in \mathcal{M}\left(\mathbb{C}^{*}\right)$ defined by
\begin{gather*}
e_{q, \lambda_{q}}(Q)=\frac{\theta_{q}(Q)}{\theta_{q}(\lambda_{q} Q)} \in \mathcal{M}(\mathbb{C}^{*}).	
\end{gather*}
\end{Definition}

The function $e_{q,\lambda_q}$ satisfies the	 $q$-difference equation $q^{Q \partial_{Q}} e_{q, \lambda_{q}}(Q)=\lambda_{q} e_{q, \lambda_{q}}(Q)$.

\begin{Definition}
The $q$-logarithm is the function $\ell_{q} \in \mathcal{M}\left(\mathbb{C}^{*}\right)$ defined by
\begin{gather*}
\ell_{q}(Q)=-Q\frac{ \theta_{q}^{\prime}(Q)}{\theta_{q}(Q)}.	
\end{gather*}	
\end{Definition}

By a little computation, one could know that the function $\ell_q $ satisfies the following $q$-difference equation
\begin{gather*}
q^{Q \partial_{Q}} \ell_{q}(Q)=\ell_{q}(Q)+1	.
\end{gather*}

Now we can state the existence of a fundamental solution for regular singular $q$-difference equations under certain conditions.

For a $q$-difference system $q^{Q\partial_Q}X_q(Q) = A_q(Q)X_q(Q) $, without loss of generality, we assume $A_q(0) \in {\rm GL}_n(\mathbb{C}) $ and moreover that it is non-resonant. We can recursively build a gauge transform $F_{q} \in {\rm GL}_{n}(\mathbb{C}(\{ Q \}))$ which sends the matrix $A_q(0)$ to the constant matrix $A_q(Q)$, for details, see~\cite[Corollary 3.2.4]{hardouin:hal-01959879}. Then we take the Jordan--Chevalley decomposition of $A_q(0)=A_s A_u$, where $A_s$ is semi-simple, $A_u$ is unipotent and $A_s$, $A_u$ commute.

Since $N=A_u-I_n$ is nilpotent, we can define
\begin{gather}
A_u^{\ell_q} := (I_n+N)^{\ell_q} := \sum_{k \geq 0 } \binom{\ell_q}{k} N^k,	 \label{A_u^l}
\end{gather}
where
\begin{gather*}
	\binom{\ell_q}{k} := \frac{\ell_q(\ell_q-1)\cdots(\ell_q-(k-1))}{k!}.
\end{gather*}
Note that (\ref{A_u^l}) is actually a finite sum and $A_u^{\ell_q}$ is unipotent, and we have
\begin{gather*}
q^{Q\partial_Q} A_u^{\ell_q} = A_u	A_u^{\ell_q} = A_u^{\ell_q} A_u.
\end{gather*}
Thus we set
\begin{gather*}
e_{q,A_u} := A_u^{\ell_q}.	
\end{gather*}
Take a basis change $P$ to diagonalise $A_s=P^{-1}\operatorname{diag}(\lambda_i)P$. We define
\begin{gather}
e_{q, A_s}:=P^{-1} \operatorname{diag}(e_{q, \lambda_{i}}(Q)) P,	 \label{e_q}
\end{gather}
which satisfies
\begin{gather*}
q^{Q\partial_Q} e_{q, A_s} = A_s e_{q, A_s} = e_{q, A_s} A_s	.
\end{gather*}
Then one can check that the product $F_{q} e_{q, A_s} e_{q, A_u}=: \mathcal{X}_{q}(Q)$ is a fundamental solution of the $q$-difference system $q^{Q \partial_{Q}} X_{q}(Q)=A_{q}(Q) X_{q}(Q)$. We arrive at the following theorem.

{\sloppy\begin{Proposition}[{\cite[Theorem 3.3.1]{hardouin:hal-01959879}}] 
The $q$-difference system $\sigma_{q} X_q(Q)=A_q(Q) X_q(Q)$, regu\-lar singular at $Q=0$, admits a fundamental matricial solution $\mathcal{X}:=M e_{q, C} \in {\rm GL}_{n}(\mathcal{M}(\mathbb{C}^{*}, 0))$, where $C \in {\rm GL}_{n}(\mathbb{C})$ and where $M \in G L_{n}(\mathbb{C}(\{ Q \}))$. The $e_{q, C}$ is defined by Jordan--Chevalley decomposition of $C$ as above.
\end{Proposition}}

\begin{Remark}
Let $A,P \in {\rm GL}_n(\mathbb{C})$, one can check that $e_{q,PAP^{-1}} = Pe_{q,A}P^{-1} $. Thus, (\ref{e_q}) is independent of the choice of $P$.	
\end{Remark}

\subsection[Monodromy of regular singular q-difference equations]
{Monodromy of regular singular $\boldsymbol {q}$-difference equations}

\begin{Definition} \label{def-fuchsian}
A $q$-difference system $q^{Q \partial_{Q}} X_{q}(Q)=A_{q}(Q) X_{q}(Q)$ is called fuchsian if it is regular singular both at $Q=0$ and $Q=\infty$.	
\end{Definition}
It is easy to see the difference equation (\ref{difference-equation}) is not fuchsian since it is not regular singular at $Q=0$. But we will see it is regular singular at~$Q=\infty$ (see (\ref{deg-25-infty})).

\begin{Definition}
Let $q^{Q \partial_{Q}} X_{q}(Q)=A_{q}(Q) X_{q}(Q)$ be a fuchsian $q$-difference system. This $q$-difference system admits a fundamental solution $\mathcal{X}_{0}(Q)$ at $Q=0$	and a second one $\mathcal{X}_\infty(1/Q)$ at $Q=\infty$. Birkhoff's connection matrix (or $q$-monodromy) $P_q$ is the ratio
\begin{gather*}
M_{q}(Q)= (\mathcal{X}_{\infty}(1 / Q) )^{-1}\mathcal{X}_{0}(Q).
\end{gather*}
\end{Definition}

Since the connection matrix relates two fundamental matrix solutions. It is invariant by difference operator $q^{Q\partial_Q}$, i.e.,
\begin{gather*}
M_q(Q) \in {\rm GL}_{n} (\mathcal{M} (\mathbb{E}_{q} ) ).
\end{gather*}
However, it is not well defined: it depends on the choice of fundamental matrix solutions. To~get rid of this dependence, we need to consider the following triple.
\begin{Definition}
A Birkhoff connection triple is a triple
\begin{gather*}
\big( A^{(0)}, M_q, A^{(\infty)} \big)	\in {\rm GL}_n(\mathbb{C}) \times {\rm GL}_n(\mathbb{E}_{q}) \times {\rm GL}_n(\mathbb{C})
\end{gather*}
up to certain equivalent. Where $A^{(0)}$ and $A^{(\infty)}$ are related to the fundamental solutions at $Q=0$ and $Q=\infty$ respectively, for more details, see \cite[p.~133]{hardouin:hal-01959879}.
\end{Definition}

The data of Birkhoff's connection triples classifies fuchsian $q$-difference systems up to gauge transformations.

\begin{Proposition}[{\cite[Theorem 3.4.9]{hardouin:hal-01959879}}]
	Rational classes $($under rational equivalence, i.e., over field $\mathbb{C}(Q))$ of fuchsian rational systems are in bijection with equivalence classes of Birkhoff connection triples.
\end{Proposition}

\subsection[Confluence of regular singular q-difference equations]
{Confluence of regular singular $\boldsymbol{q}$-difference equations}

First, let us introduce some interesting formulas we needed when considering the confluence of difference equations. We fix $\tau_{0}$ such that $\operatorname{Im} \tau_{0}>0$ and $q_{0}:={\rm e}^{-2 {\rm i} \pi \tau_{0}}$ and $|q_0|<1$. This defines a~discrete logarithmic spiral $q_{0}^{\mathbb{Z}}:=\left\{q_0^{k} \mid k \in \mathbb{Z}\right\} \subset \mathbb{C}$ and a continuous spiral $q_0^{\mathbb{R}}:=\left\{q_0^{k} \mid k \in \mathbb{R}\right\}\allowbreak \subset \mathbb{C}$. Let $\Omega = \mathbb{C}^* - q_0^{\mathbb{R}}$. Denote by $\log (Q)$ the logarithm on $\Omega$ such that $1 \mapsto 0$. Let $Q^\mu := {\rm e}^{\mu \log (Q)}$.

\begin{Lemma}[{\cite[Section 3.1.7, Corollaire 1]{Sauloy00}}]
Let $q(t)=q^t_0$, $t\in (0,1]$. Assume there exist complex numbers $\alpha_0, \alpha_1 \in \mathbb{C}$ so that $Q_i(q(t))=Q_0q_0^{\alpha_it+o(t)}$, $Q_0 \in \Omega$. Then, on $\Omega$, we have the uniform convergence when $t \rightarrow 0$
	\begin{gather*}
	\lim_{q \rightarrow 1} \frac{\theta_{q(t)}(Q_1(q(t)) )}{\theta_{q(t)}(Q_2( q(t)))} = Q_0^{\alpha_2-\alpha_1}	.
	\end{gather*}
\end{Lemma}
\begin{Proposition}[{\cite[Sections 3.1.3 and 3.1.4]{Sauloy00}}] \label{confluence-ell-e}
	As the above notation, consider $\lambda_{q(t)}$, $\mu \in \mathbb{C}^*$ such that $\frac{\lambda_{q(t)}-1}{q-1} \rightarrow \mu$. Then we have the asymptotics:
	\begin{itemize}\itemsep=0pt
	\item[$1.$] We have the uniform convergence on any compact of $\Omega$	
	\begin{gather*}
	\lim_{t \rightarrow 0}(q(t)-1)\ell_{q(t)}(-Q) = \log(Q).	
	\end{gather*}
 \item[$2.$] We have the uniform convergence on any compact of $\Omega$	
 \begin{gather*}
 	\lim_{t \rightarrow 0} e_{q(t),\lambda_{q(t)}}(-Q) = Q^\mu.
 \end{gather*}
	\end{itemize}
\end{Proposition}

Now, let's introduce the definition of confluence.
\begin{Definition}[{\cite[Section 3.2]{Sauloy00}}]
	Let $q(t)=q^t_0$, for $t \in (0,1] $. A regular singular, non $q$-resonant difference system $q^{Q\partial_Q}X_q(Q)=A_q(Q)X_q(Q)$ is said to be confluent if it satisfies four conditions below. Set $B_q(Q) = \frac{A_q(Q)-Id}{q-1} $, whose coefficients have poles $Q_1(q), \ldots, Q_k(q)$ in the input $Q$. We require that
	\begin{itemize}\itemsep=0pt
	\item[1.] The $q$-spirals satisfy $ \bigcap_{i=1}^k Q_i(q_0)q_0^{\mathbb{R}} = \varnothing $.
	\item[2.] There exists a matrix $\tilde{B} \in GL_n(\mathbb{C}(Q))$ such that
	\begin{gather*}
	\lim_{t \rightarrow 0} B_{q(t)} = \widetilde{B},
	\end{gather*}
	uniformly in $Q$ on any compact of $\mathbb{C}^*- \bigcup_{i=0}^k Q_iq_0^{\mathbb{R}} $, set $Q_0=1$.
	\item[3.] This limit defines a regular singular, non resonant differential system
 \begin{gather*}
	Q\frac{\rm d}{{\rm d}Q} \widetilde{X} = \widetilde{B} \widetilde{X}.	
	\end{gather*}
 \item[4.] There exists, for each $t$, a Jordan decompositions $B_{q(t)}(0) = P_{q(t)}^{-1} J_{q(t)} P_{q(t)} $ as well as $\widetilde{B}(0)=\widetilde{P}^{-1}\widetilde{J}\widetilde{P}$. We ask that
 \begin{gather*}
 	\lim_{t\rightarrow 0}P_{q(t)} = \widetilde{P}.
 \end{gather*}
	\end{itemize}
\end{Definition}

If the difference system is confluent, then there is a confluence of the solutions.
\begin{Proposition}[{\cite[Theorem V.2.4.7]{Ro19}}]
Let $q(t)=q^t_0$, for $t \in [0,1] $. Consider a regular singular confluent $q$-difference system $q^{Q \partial_{Q}} X_{q}(Q)=A_{q}(Q) X_{q}(Q)$, whose limit system is $Q \partial_{Q} \widetilde{X}(Q)=$ $\widetilde{B}(Q) \widetilde{X}(Q)$.

Assume that there exists a vector $X_{0} \in \mathbb{C}^{n}\backslash \{0\}$, independent of $q$, such that $A_{q(t)} X_{0}=X_{0}$ for all $t \in(0,1]$. We also assume that we have a solution $\mathcal{X}_q(Q)$ of the $q$-difference system satisfying the initial condition $\mathcal{X}_q(0)=X_{0}$.

Let $\widetilde{\mathcal{X}}(Q)$ be the unique solution of $Q \partial_{Q} X(Q)=\widetilde{B}(Q) X(Q)$ satisfying the initial condition $\widetilde{\mathcal{X}}(0)=X_{0}$. We have
\begin{align*}
\lim _{t \rightarrow 0} \mathcal{X}_{q(t)}(Q)=\widetilde{\mathcal{X}}(Q)	
\end{align*}
uniformly in $Q$ on any compact of $\mathbb{C}^{*}-\bigcup_{i=0}^{k} Q_{i} q_{0}^{\mathbb{R}}$.
\end{Proposition}

\section{The difference equation for quintic}\label{sec3}
\subsection{General technique: Newton polygon}
Let's consider the equation
\begin{gather*}
\sum_{i=0}^n a_i(Q) (\sigma_q)^i f(Q)=0,	
\end{gather*}
with
\begin{gather*}
a_i(Q)= a_{i,0}+a_{i,1}Q+a_{i,2}Q^2 + \cdots.	
\end{gather*}
We call the following equation the \emph{characteristic equation}	
\begin{gather}
a_{n,0}x^n+ a_{n-1,0} x^{n-1} + \cdots + a_{1,0}x + a_{0,0} =0, \label{char-equ}
\end{gather}
which plays an important role in constructing solutions.

Denote by $a_{i,j_i}$ the first nonzero coefficient in $a_i(Q)$, and choosing $i$- and $j$-axes as horizontal and vertical axes respectively, plot the points $(n-i, j_i)$. Construct a broken line, convex downward, such that both ends of each segment of the line are points of the set $(n-i, j_i)$. Then we obtain a Newton polygon as follows
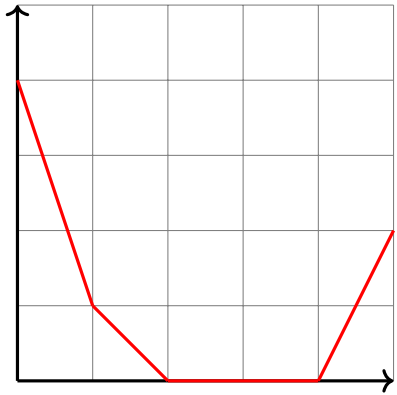
\begin{center}
\begin{tikzpicture}
\draw [xstep=1, ystep=1, draw=gray](0,0)grid(5,5);
\draw [->, very thick](0,0)--(5,0);
\draw [->, very thick](0,0)--(0,5);

\draw [very thick, draw=red](0,4)--(1,1);
\draw [very thick, draw=red](1,1)--(2,0);
\draw [very thick, draw=red](2,0)--(4,0);
\draw [very thick, draw=red](4,0)--(5,2);
\end{tikzpicture}
\end{center}
Note that the horizontal segment corresponds to the characteristic equation	
\begin{gather*}
a_{k,0}x^k+ a_{k-1,0} x^{k-1} + \cdots + a_{d,0}x^d = 0.
\end{gather*}
The degree of the above characteristic equation is 1 less than the number of points on or above that segment.

\begin{Example}
	Consider the following equation:
	\begin{gather*}
	\big[	\big(Q^4+2Q^7\big)\sigma_q^6 + \big(Q+3Q^5\big)\sigma_q^5 + \big(3+2Q^3\big)\sigma_q^4 + 2 \sigma_q^3 + 3Q \sigma_q^2 + Q^2\sigma_q \big]f(Q)=0.
	\end{gather*}
Then the associated Newton polygon is
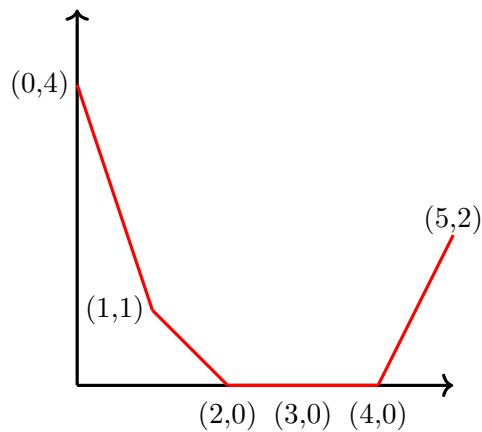
\begin{center}
\begin{tikzpicture}
\draw [xstep=1, ystep=1, draw=gray];
\draw [->, very thick](0,0)--(5,0);
\draw [->, very thick](0,0)--(0,5);

\draw [very thick, draw=red](0,4)--(1,1);
\draw [very thick, draw=red](1,1)--(2,0);
\draw [very thick, draw=red](2,0)--(4,0);
\draw [very thick, draw=red](4,0)--(5,2);
\draw [draw=red](-0.5,4) node{(0,4)};
\draw [draw=red](0.5,1) node{(1,1)};
\draw [draw=red](2,-0.4) node{(2,0)};
\draw [draw=red](3,-0.4) node{(3,0)};
\draw [draw=red](4,-0.4) node{(4,0)};
\draw [draw=red](5,2.2) node{(5,2)};
\end{tikzpicture}
\end{center}
\end{Example}

The general technique to construct solutions is as follows:
\begin{itemize}\itemsep=0pt
\item{\it{Horizontal segment:}} As mentioned above, it corresponds to characteristic equation. Using the non-zero roots, we could construct the associated solutions as regular singular cases.
\item{\it{Non-horizontal segment:}} For each non-horizontal segment of slope $\mu$, a rational number.
\begin{itemize}\itemsep=0pt
\item If $\mu = r$ is an integer, we consider a formal series solution of the form
\begin{gather*}
	\theta^{r}_q(Q)\sum_{n=0}^{\infty} f_n(q)Q^n.
\end{gather*}

\item If $\mu = t/s$ is a rational number with $s$ positive, then we consider a formal series solution of the form
\begin{gather*}
 \theta_{q^{t/s}}\big(Q^{t/s}\big) \sum_{n=0}^{\infty} f_{n}(q) Q^{n/s}.	
\end{gather*}
\end{itemize}
\end{itemize}
For more details, see \cite{Adams2, Adams1931, RW21}.

\begin{Remark}\label{remark3.2}
In Adams' works, he used $q^{\frac{\mu}{2}(t^2-t)}$, where $t=\frac{\ln Q}{\ln q}$.	
\end{Remark}

\subsection[Solutions at Q=0]{Solutions at $\boldsymbol{Q=0}$}

The $K$-theoretic $I$-function of quintic is as follows \cite{givental2015}
\begin{gather*}
I^K = P^{l_q(Q)} \sum_{d \geq 0} \frac{\prod_{k=1}^{5d}\big(1-P^5q^k\big)}{\prod_{k=1}^{d}\big(1-Pq^k\big)^5}Q^d,	
\end{gather*}
it satisfies the following degree 25 difference equation
\begin{gather}
\Bigg[\big( 1- q^{Q\partial_Q} \big)^5 - Q \prod_{k=1}^{5}\big( 1-q^k q^{5Q\partial_Q}\big)\bigg] I^K=0	 \mod\big((1-P)^5\big).	\label{deg-25-difference-equation}
\end{gather}

\begin{Remark}See \cite{GS21} for additional discussion on $q$-deformed Picard--Fuchs equation.
\end{Remark}

The characteristic equation at $Q=0$ is
\begin{gather*}
(1-x)^5=0.	
\end{gather*}
We can construct only 5 solutions by expanding $I^K_X(q,Q)$ with respect to the $K$-group basis $(1-P)^k$, $k=0,1,2,3,4$. Next we use Adams' method to find other solutions at $Q=0$. The Newton's polygon of the above difference equation (\ref{deg-25-difference-equation}) is as follows
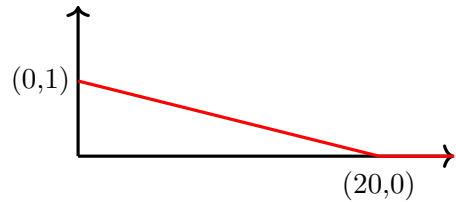
\begin{center}
\begin{tikzpicture}
\draw [xstep=1, ystep=1, draw=gray];
\draw [->, very thick](0,0)--(5,0);
\draw [->, very thick](0,0)--(0,2);
\draw [very thick, draw=red](0,1)--(4,0);
\draw [very thick, draw=red](4,0)--(5,0);
\draw [draw=red](-0.5,1) node{(0,1)} ;
\draw [draw=red](4,-0.4) node{(20,0)};
\end{tikzpicture}
\end{center}
Then we have solutions at $Q=0$ of the form
\begin{gather*}
	\theta_{q^{\frac{1}{20}}}\big(Q^{\frac{1}{20}}\big)G\big(Q^{\frac{1}{20}}\big) = \theta_{q^{\frac{1}{20}}}\big(Q^{\frac{1}{20}}\big) \sum_{d=0}^\infty g_d Q^{\frac{d}{20}},
\end{gather*}
where $\theta_{q^{\frac{1}{20}}}\big(Q^{\frac{1}{20}}\big)$ satisfies
\begin{gather*}
q^{Q\partial_Q} \theta_{q^{\frac{1}{20}}}\big(Q^{\frac{1}{20}}\big) = \bigg(\frac{1}{Q}\bigg)^{\frac{1}{20}}\theta_{q^{\frac{1}{20}}}\big(Q^{\frac{1}{20}}\big),	
\end{gather*}
then
\begin{gather*}
	q^{5Q\partial_Q} \theta_{q^{\frac{1}{20}}}\big(Q^{\frac{1}{20}}\big) = q^{-\frac{1}{2}}\bigg(\frac{1}{Q}\bigg)^{\frac{1}{4}}\theta_{q^{\frac{1}{20}}}\big(Q^{\frac{1}{20}}\big).
\end{gather*}
Substituting into (\ref{deg-25-difference-equation}), we find that $G\big(Q^{\frac{1}{20}}\big)$ satisfies the following difference equation:
\begin{gather*}
\Big[\big(Q^{\frac{1}{20}}-q^{-\frac{4}{20}}q^{Q\partial_Q}\big) \big(Q^{\frac{1}{20}}-q^{-\frac{3}{20}}q^{Q\partial_Q} \big) \big(Q^{\frac{1}{20}}-q^{-\frac{2}{20}}q^{Q\partial_Q} \big) \big(Q^{\frac{1}{20}}-q^{-\frac{1}{20}}q^{Q\partial_Q}\big)\big(Q^{\frac{1}{20}}-q^{Q\partial_Q}\big)
\\ \qquad
{}- \big(Q^{\frac{1}{4}}\!-q^{-\frac{2}{4}}q^{5Q\partial_Q}\big) \big(Q^{\frac{1}{4}}\!-q^{-\frac{1}{4}}q^{5Q\partial_Q}\big) \big(Q^{\frac{1}{4}}\!-q^{5Q\partial_Q}\big) \big(Q^{\frac{1}{4}}\!-q^{\frac{1}{4}}q^{5Q\partial_Q}\big) \big(Q^{\frac{1}{4}}\!-q^{\frac{1}{2}}q^{5Q\partial_Q}\big)\Big]
\\ \qquad
{}\times G\big(Q^{\frac{1}{20}}\big) =0.
\end{gather*}
Let $z=Q^{\frac{1}{20}}$, $p=q^{\frac{1}{20}}$ and $\sigma_p=p^{z\partial_z}$, then the above difference equation takes the following form
\begin{gather}
\big[\big(z-p^{-4}\sigma_p\big)\big(z-p^{-3}\sigma_p\big)\big(z-p^{-2}\sigma_p\big)\big(z-p^{-1}\sigma_p\big)(z-\sigma_p) \nonumber
\\ \qquad
{}- \big(z^5-p^{-10}\sigma_p^5\big)\big(z^5-p^{-5}\sigma_p^5\big)\big(z^5-\sigma_p^5\big)\big(z^5-p^{5}\sigma_p^5\big)\big(z^5-p^{10}\sigma_p^5\big) \big] G(z)=0. \label{deg-20-difference-equation}
\end{gather}
For $z=0$, we obtain the characteristic equation
\begin{gather*}
 p^{-10}x^5- x^{25} = 0,	
\end{gather*}
i.e.,
\begin{gather*}
x=\xi p^{-\frac{1}{2}},\qquad \text{for} \quad \xi^{20}=1. 	
\end{gather*}
Consider a solution of the form
\begin{gather*}
	e_{p,\xi p^{-1/2}}(z)F(z) = e_{p,\xi p^{-1/2}}(z)\sum_{n \geq 0}f_nz^n,
\end{gather*}
then $F(z)$ satisfies
\begin{gather*}
\Big[\big(z-\xi p^{-\frac{9}{2}}\sigma_p\big) \big(z-\xi p^{-\frac{7}{2}}\sigma_p\big)
\big(z-\xi p^{-\frac{5}{2}}\sigma_p\big) \big(z-\xi p^{-\frac{3}{2}}\sigma_p\big)
\big(z-\xi p^{-\frac{1}{2}} \sigma_p\big)
\\ \qquad
{}- \big(z^5-\xi^5p^{-\frac{25}{2}}\sigma_p^5\big)\big(z^5-\xi^5 p^{-\frac{15}{2}}\sigma_p^5\big) \big(z^5-\xi^5p^{-\frac{5}{2}}\sigma_p^5\big) \big(z^5-\xi^5p^{\frac{5}{2}}\sigma_p^5\big) \big(z^5-\xi^5p^{\frac{15}{2}}\sigma_p^5\big) \Big]
\\ \qquad
{}\times F(z)=0.
\end{gather*}
After a short computation, we expand the above difference equation as follows
\begin{gather}
\Big[ \xi^5p^{-\frac{25}{2}}\big(\sigma_p^{25}-\sigma_p^{5}\big)+5\xi^4p^{-8}z\sigma_p^4-10\xi^3p^{-\frac{9}{2}} z^2\sigma_p^3 + 10 \xi^2p^{-2}z^3\sigma_p^2-5\xi p^{-\frac{1}{2}}z^4\sigma_p\nonumber
\\ \qquad
{}+ z^5\big( 1-\big(1+p^{\frac{15}{2}}+p^{\frac{55}{2}}+p^{\frac{95}{2}}+p^{\frac{135}{2}} \big)\sigma_p^{20} \big) \nonumber
\\ \qquad
{}+\xi^{15}\big(p^{\frac{15}{2}}+p^{\frac{55}{2}}+2p^{\frac{95}{2}}+2p^{\frac{135}{2}} +2p^{\frac{175}{2}}+p^{\frac{215}{2}}+p^{\frac{255}{2}} \big)z^{10}\sigma_p^{15} \nonumber
\\ \qquad
{}-\xi^{10}\big(p^{10}+p^{30}+2p^{50}+2p^{70}+2p^{90}+p^{110}+p^{130} \big)z^{15}\sigma_p^{10} \nonumber
\\ \qquad
{}+ \xi^{5}\big(p^{\frac{15}{2}}+p^{\frac{55}{2}} +p^{\frac{95}{2}}+p^{\frac{135}{2}}+p^{\frac{175}{2}} \big)z^{20}\sigma_p^{5} -z^{25} \Big]F(z)=0. \label{expansion-20}
\end{gather}
Thus we arrive at the following proposition.

\begin{Proposition}
The exceptional $20$ solutions of \eqref{deg-20-difference-equation} are given as follows: for each $20$th root of unity $\xi$, we have a solution of form
\begin{gather*}
	e_{p,\xi p^{-1/2}}(z)F(z) = e_{p,\xi p^{-1/2}}(z)\sum_{n \geq 0}f_nz^n,
\end{gather*}
where $F(z)$ satisfies
\begin{gather*}
\Big[ \big(z-\xi p^{-\frac{9}{2}}\sigma_p\big)\big(z-\xi p^{-\frac{7}{2}}\sigma_p\big)\big(z-\xi p^{-\frac{5}{2}}\sigma_p\big)\big(z-\xi p^{-\frac{3}{2}}\sigma_p\big)\big(z-\xi p^{-\frac{1}{2}} \sigma_p\big)
\\ \qquad
{}- \big(z^5-\xi^5p^{-\frac{25}{2}}\sigma_p^5\big)\big(z^5-\xi^5 p^{-\frac{15}{2}}\sigma_p^5\big)\big(z^5-\xi^5p^{-\frac{5}{2}}\sigma_p^5\big) \big(z^5-\xi^5p^{\frac{5}{2}}\sigma_p^5\big) \big(z^5-\xi^5p^{\frac{15}{2}}\sigma_p^5\big) \Big]
\\ \qquad
{}\times F(z)=0.
\end{gather*}
\end{Proposition}

\begin{proof}
From (\ref{expansion-20}), only $f_0$ is a free variable, and $f_n$ is determined by $\{f_k \}_{k \leq n}$. Then we obtain~20 solutions.	
\end{proof}

\begin{Remark}
The exceptional solutions are linearly independent over $\mathcal{M}(\mathbb{E}_q)$, since their Wronskian matrix (see~\cite[Lemma 2.3.3]{hardouin:hal-01959879}) does not equal to 0. Actually, it is sufficient to see the Wronskian matrix for $\{e_{p,\xi^k p^{-1/2}} \}_{k=0}^{19}$ and it turns out to be the Vandermonde matrix.	
\end{Remark}

\subsection[Solutions at Q=infty]
{Solutions at $\boldsymbol{Q=\infty}$}

Let $w=1/Q$, (\ref{deg-25-difference-equation}) becomes
\begin{gather}
\Bigg[ \prod_{k=1}^{5}\big(1 - q^{-k}q^{5w\partial_w}\big) -q^{10}w q^{20w\partial_w} \big(1 - q^{w\partial_w} \big)^5 	\Bigg] F(w)=0,	
\label{deg-25-infty}
\end{gather}
which is regular singular. The characteristic equation is as follows
\begin{gather*}
\prod_{k=1}^{5}\big(1-q^{-k}x^5\big)=0,	
\end{gather*}
with 25 distinct roots
\begin{gather*}
	 q^{\frac{l}{5}}\xi^m, \qquad l=1,\dots,5, \quad	m=0,\dots,4.
\end{gather*}
Here $\xi$ is the fifth root of unity. For each root $q^{\frac{l}{5}}\xi^m$, we construct the following solution
\begin{gather*}
W_{l,m}(w) = e_{q,q^{\frac{l}{5}}\xi^m}(w) \sum_{d \geq 0} f_d w^d.	
\end{gather*}
Substituting the above formula into (\ref{deg-25-infty}), since
\begin{gather*}
q^{w\partial_w} e_{q,q^{\frac{l}{5}}\xi^m}(w)=	q^{\frac{l}{5}}\xi^m e_{q,q^{\frac{l}{5}}\xi^m}(w),
\end{gather*}
then $\sum_{d \geq 0} f_d w^d$ satisfies the following difference equation
\begin{gather*}
\Bigg[ \prod_{k=1}^{5}\big( 1 - q^{l-k}q^{5w\partial_w} \big) -q^{10+4l}w q^{20w\partial_w} \big(1 -	q^{\frac{l}{5}}\xi^m q^{w\partial_w}\big)^5 \Bigg] G(w)=0.	
\end{gather*}
{\samepage Then one obtain 25 solutions at $Q=\infty$ as follows
\begin{align*}
W_{l,m}(w) &= e_{q,q^{\frac{l}{5}}\xi^m}(w)\sum_{d \geq 0} \frac{\prod_{k=0}^{d-1}\big(1-\xi^{-m}q^{-k-\frac{l}{5}}\big)^5}{\prod_{k=0}^{5d-1} \big(1-q^{-k-l}\big) } w^d
\\
 &= e_{q,q^{\frac{l}{5}}\xi^m}(1/Q) \sum_{d\geq0}\frac{\prod_{k=0}^{d-1}\big(1-\xi^{-m}q^{-k-\frac{l}{5}}\big)^5}{\prod_{k=0}^{5d-1}\big(1-q^{-k-l}\big)}Q^{-d}. 
\end{align*}}
Since we require $|q|<1$ and
\begin{gather*}
\lim_{n \rightarrow \infty}	\frac{\prod_{k=0}^{(n+1)-1}\big(1-\xi^{-m}q^{-k-\frac{l}{5}}\big)^5}{\prod_{k=0}^{5(n+1)-1}\big(1-q^{-k-l}\big)} \frac{\prod_{k=0}^{5n-1}\big(1-q^{-k-l}\big)}{\prod_{k=0}^{n-1}\big(1-\xi^{-m}q^{-k-\frac{l}{5}}\big)^5} =\lim_{n \rightarrow \infty} \frac{\big(1-\xi^{-m}q^{-n-\frac{l}{5}}\big)^5}{\prod_{k=5n}^{5n+4}\big(1-q^{-k-l}\big)} = 0.
\end{gather*}
The 25 solutions are convergent.
\begin{Remark}
These 25 solutions may relate to K-theoretic FJRW theory, for hints, see \cite{GDZ21}.	
\end{Remark}

\section[Auxiliary q-series and analytic continuation]{Auxiliary $\boldsymbol q$-series and analytic continuation}\label{sec4}

In this section, we construct a $K$-group valued $q$-series, which is a generalization of the series~(\ref{K-I-function}), and it satisfies a difference equation like (\ref{deg-25-difference-equation}). Besides, we use Mellin--Barnes--Watson method to relate the solutions at $Q=0$ and $Q=\infty$.

\subsection[Auxiliary q-series]{Auxiliary $\boldsymbol q$-series}

We construct the following $K\big(\mathbb{P}^{n-1}\big)$ valued $q$-series motivated by \cite{CR13}:
\begin{gather}
F_{m,n}(\vec{\alpha},Q)=P^{l_q(Q)}\sum_{d=0}^{\infty} \frac{\prod_{i=1}^{m}(P{\alpha_i};q)_d}{(Pq;q)_d^n} Q^d, \qquad m \geq n. \label{F_{m,n}}
\end{gather}
Since $(1-P)^n=0$ in $K\big(\mathbb{P}^{n-1}\big)$, then it satisfies the following difference equation:
\begin{gather}
\Bigg[\big(1-q^{Q\partial_Q}\big)^n-Q\prod_{i=1}^{m}\big(1-{\alpha_i}q^{Q\partial_Q}\big) \Bigg] F_{m,n}(Q)=0 \mod\big((1-P)^n\big). \label{aux-dif-F_{m,n}}
\end{gather}
Suppose $\alpha_i \notin \alpha_j q^{\mathbb{Z}\backslash \{0\}}$, i.e., the difference equation is ($q$-)non-resonant, the characteristic equation at $Q=\infty$ is as follows
\begin{gather*}
\prod_{i=1}^m\big(1-\alpha^{-1}_ix\big)=0,	
\end{gather*}
with $m$-distinct roots $\{\alpha_i \}_{i=1}^m$, the same as the discussion in previous section, for each root $ \alpha_i$, we construct a solution as follows
\begin{gather*}
W_i(1/Q) = e_{q,\alpha_i}(1/Q) \sum_{d \geq 0} f_d Q^{-d}.	
\end{gather*}
Recall
\begin{gather*}
	q^{Q\partial_Q} e_{q,\alpha_i}(1/Q) = \alpha_i^{-1} e_{q,\alpha_i}(1/Q),
\end{gather*}
then $ \sum_{d \geq 0} f_d Q^{-d}	 $ satisfies the following difference equation
\begin{gather*}
\Bigg[\big(1-\alpha_i^{-1}q^{Q\partial_Q}\big)^n-Q\prod_{j=1}^{m}\big(1-{\alpha_j}/\alpha_i q^{Q\partial_Q}\big)\Bigg] G\big(Q^{-1}\big) = 0.
\end{gather*}
Thus, we obtain
\begin{align*}
W_j(1/Q) ={}&	e_{q,\alpha_j}(1/Q) \sum_{d \geq 0} \frac{\prod_{k=0}^{d-1}(1-\alpha_j q^k )^n q^{(m-n)d(d-1)/2} }{\prod_{i=1}^m \prod_{k=1}^{d} (1- \alpha_{j}/\alpha_{i} q^{k}) }
\\
&\times
\Bigg((-1)^{m-n} \Bigg(\prod_{i=1}^m\alpha_j/\alpha_i\Bigg) \alpha_j^{-n} q^m /Q \Bigg)^d
\\
={}&e_{q,\alpha_j}(1/Q) \sum_{d \geq 0} \frac{\prod_{k=0}^{d-1}\big(1-\alpha_j^{-1}q^{-k} \big)^n }{\prod_{i=1}^m \prod_{k=1}^{d} \big(1-\alpha_{i}/\alpha_j q^{-k}\big) } Q^{-d}.
\end{align*}

\begin{Remark}
The above solutions are linearly independent over $\mathcal{M}(\mathbb{E}_q)$. The same reason as Remark~\ref{remark3.2}.
\end{Remark}

\begin{Remark}
If we take $n=5$, $m=25$ and $\{ \alpha_i \}_{i=1}^{25}=\big\{ \xi^l q^{\frac{k}{5}} \mid k,l=1,2,3,4,5 \big\}$, then
\begin{gather*}
F_{25,5}\big(\big\{\xi^l q^{\frac{k}{5}} \big\},Q\big) = P^{l_q(Q)} \sum_{d \geq 0} \frac{\prod_{k=1}^{5d}\big(1-P^5q^k\big)}{\prod_{k=1}^{d}\big(1-Pq^k\big)^5}Q^d,
\end{gather*}
and for $\alpha_j=q^{\frac{l}{5}\xi^m}$, we have
\begin{align*}
W_j(1/Q) &= e_{q,q^{\frac{l}{5}}\xi^m}(1/Q)	 \frac{\prod_{k=0}^{d-1}\big(1-q^{-\frac{l}{5}}\xi^{-m}q^{-k} \big)^n }{ \prod_{k=1}^{d} \big(1-q^{1-l} q^{-5k}\big)\cdots \big(1-q^{5-l} q^{-5k}\big) } Q^{-d}
\\
 &= e_{q,q^{\frac{l}{5}}\xi^m}(1/Q) \sum_{d\geq0}\frac{\prod_{k=0}^{d-1}\big(1-\xi^{-m}q^{-k-\frac{l}{5}}\big)^5}{\prod_{k=0}^{5d-1}\big(1-q^{-k-l}\big)}Q^{-d} \\
 &= W_{l,m}(1/Q).
\end{align*}
\end{Remark}

\subsection{Analytic continuation}
For the sake of simplicity, we shall assume in this section that $0 < q < 1$ and write
\begin{gather*}
q={\rm e}^{-w}, \qquad w > 0.	
\end{gather*}
The results can be extended to complex $q$ in the unit disc using analytic continuation.

Consider the following contour integral. We follow the argument of \cite[pp.~115--118]{GR04} to show that this integral is well defined. For $|Q|<1$, we can close the contour to the right, it equals to~(\ref{F_{m,n}}),
\begin{gather}
P^{l_q(Q)} \frac{\prod_{i=1}^m(P\alpha_i;q)_\infty}{(Pq;q)_{\infty}^n} \int_{C} \frac{\big(Pq^{s+1};q\big)_{\infty}^n}{\prod_{i=1}^m(P\alpha_i q^s;q)_\infty} \frac{\pi (-Q)^s}{\sin \pi s} \frac{{\rm d}s}{-2\pi {\rm i}}. 	\label{counter-integral}
\end{gather}
Here we view $P={\rm e}^{-H}$. Although $H$ is the hyperplane class, we consider it as a formal variable valued in $\mathbb{C}$. $C$ is a curve from $-{\rm i} \infty$ to $+{\rm i} \infty$ such that only the non-negative zeros of $\sin \pi s $ lie on the right side of $C$.

By the triangle inequality,
\begin{gather*}
\big|1-\big| a\big|{\rm e}^{-\omega \operatorname{Re}(s)}\big| \leq|1-a q^{s}| \leq 1+|a| {\rm e}^{-\omega \operatorname{Re}(s)},	
\end{gather*}
we have
\begin{gather*}
\bigg| \frac{\big(Pq^{s+1};q\big)_{\infty}^n}{\prod_{i=1}^m(P\alpha_i q^s;q)_\infty} \bigg| \leq \prod_{k=0}^\infty \frac{\big(1+|P|{\rm e}^{-(k+1+\operatorname{Re}(s) w)}\big)^n}{\prod_{i=1}^m\big(1-|P\alpha_i|{\rm e}^{-(k+\operatorname{Re}(s)w)}\big)},
\end{gather*}
which is bounded on the contour~$C$. Hence the integral (\ref{counter-integral}) converges if $|\arg (-Q)|<\pi$.

Let $C_{R_+}$ be a large clockwise-oriented semicircle of radius $R$ with a center at the origin that lies to the right of~$C$. The semicircle is terminated by $C$ and bounded away from the poles. Now consider the contour integral over $C_{R_+}$ instead of~$C$.

Setting $s=R {\rm e}^{{\rm i} \theta}$, we have for $|s|<1$ that
\begin{align*}
\operatorname{Re}\bigg[\log \frac{(-Q)^{s}}{\sin \pi s}\bigg]
&=R[\cos \theta \log |Q|-\sin \theta \arg (-Q)-\pi|\sin \theta|]+O(1) \\
&\leq-R[\sin \theta \arg (-Q)+\pi|\sin \theta|]+O(1).
\end{align*}
Hence, when $|Q|<1$ and $|\arg (-Q)|<\pi-\delta$, $0<\delta<\pi$, we have
\begin{gather*}
\frac{(-Q)^{s}}{\sin \pi s}=O[\exp (-\delta R|\sin \theta|)],
\end{gather*}
as $R \rightarrow \infty$, then the integral on $C_{R_+}$ tends to zero as $R \rightarrow \infty$. Therefore, by applying Cauchy's theorem, we can prove (\ref{counter-integral}) equals to (\ref{F_{m,n}}) through tedious computation.

Similarly, if we replace the contour $C$ by a contour $C_{R_-}$ consisting of a large counterclockwise-oriented semicircle of radius $R$ with center at the origin that lies to the left of $C$. From an asymptotic formula
\begin{gather*}
\operatorname{Re}[\log (q^{s} ; q)_{\infty}]=\frac{\omega}{2}(\operatorname{Re}(s))^{2}+\frac{\omega}{2} \operatorname{Re}(s)+O(1),	
\end{gather*}
as $R \rightarrow -\infty$. Without loss of generality, we assume $P=q^{h}$, $\alpha_i=q^{a_i}$ and let $h$, $a_i$ be real numbers.

Then
\begin{gather*}
n(1+h+ \operatorname{Re}(s))^2 + n(1+h+ \operatorname{Re}(s))-\sum_{i=1}^m \big[(a_i+h+\operatorname{Re}(s))^2 + (a_i+h+ \operatorname{Re}(s))\big]
\\ \qquad
{}= 2\bigg[n(1+h)-\sum_{i=1}^m (a_i+h) \bigg] \operatorname{Re}(s) + (n-m)
\big(\operatorname{Re}(s)+ \operatorname{Re}^2(s)\big)+ \rm{const}.
\end{gather*}
Note that $q={\rm e}^{-w}$, $w >0$, then the asymptotic formula for $(q^s;q)_{\infty} $ implies that
\begin{align*}
\frac{\big(Pq^{s+1};q\big)_{\infty}^n}{\prod_{i=1}^m(P\alpha_i q^s;q)_\infty} = O\Bigg( \bigg| q^{\frac{m-n}{2}} \frac{\prod_{i=1}^m P\alpha_i}{(Pq)^n} \bigg|^{\operatorname{Re}(s)} \Bigg),
\end{align*}
when $\operatorname{Re}(s) \rightarrow -\infty$ with s bounded away from the zeros and poles.

Similarly, it can be shown that if $|Q|$ is big enough, we can close the contour the left, i.e., the integral (\ref{counter-integral}) on $C_{R_-}$ tends to zero as $R \rightarrow -\infty$. Thus, (\ref{counter-integral}) equals to the sum of residues at
\begin{gather*}
s=w^{-1}(-H+\log \alpha_j+ l \cdot2\pi {\rm i}) -k \qquad \text{and} \qquad s=-1-h,
\end{gather*}
where $j=1, \dots, m$, $l \in \mathbb{Z}$, $n,h \in \mathbb{N}$. The residue at $s=-1-h$ contains a term
\begin{gather*}
(1-P)^n	
\end{gather*}
from
\begin{gather*}
\big(Pq^{s+1};q\big)_{\infty}^n.	
\end{gather*}
And the residue at $s=w^{-1}(-H+\log \alpha_j+ l \cdot2\pi {\rm i}) -k$ is
\begin{gather*}
\operatorname{Res}_{s=w^{-1}(-H+\log \alpha_j+ l \cdot2\pi {\rm i}) -k} \frac{\big(Pq^{s+1};q\big)_{\infty}^n}{\prod_{i=1}^m(P\alpha_i q^s;q)_\infty} \frac{\pi (-Q)^s}{\sin \pi s}
\\ \qquad
{}=\frac{\big(\alpha_j^{-1}q^{1-k};q\big)_\infty^n}{\prod_{i=1,i \neq j}^{m} \big(\alpha_i/\alpha_jq^{-k};q \big)_{\infty}} \frac{(q^{k+1};q)_{\infty}}{(q,q;q)_{\infty}} (-1)^k q^{\frac{k(k+1)}{2}}
\\ \qquad \hphantom{=}
{}\times \frac{\pi (-Q)^{w^{-1}(-H+\log \alpha_j)-k} w^{-1}}{\sin \pi\big(w^{-1}(-H+\log \alpha_j)-k+w^{-1} 2l \pi {\rm i}\big) }
\exp \big\{ 2l \pi {\rm i} w^{-1} \log (-Q) \big\}.
\end{gather*}
If we sum over $k$, we obtain
\begin{gather*}
\sum_{k=0}^\infty \frac{\big(\alpha_j^{-1}q^{1-k};q\big)_\infty^n}{\prod_{i=1,i \neq j}^{m} \big(\alpha_i/\alpha_jq^{-k};q \big)_{\infty}} \frac{\big(q^{k+1};q\big)_{\infty}}{(q,q;q)_{\infty}} q^{\frac{k(k+1)}{2}} Q^{-k} 	
\\ \qquad
{}= \frac{\big(\alpha_j^{-1}q;q\big)^n_\infty}{\prod_{i=1,i\neq j}^m(\alpha_i/\alpha_j;q)_\infty (q;q)_\infty} \sum_{k=0}^\infty \frac{\big(\alpha_j^{-1};q^{-1}\big)_k^n}{\prod_{i=1}^m\big(\alpha_i/\alpha_jq^{-1};q^{-1}\big)_k} Q^{-k}
\\ \qquad
{}= \frac{\big(\alpha_j^{-1}q;q\big)^n_\infty}{\prod_{i=1,i\neq j}^m(\alpha_i/\alpha_j;q)_\infty (q;q)_\infty } {\rm e}^{-1}_{q,\alpha_j}(1/Q) W_j(1/Q).
\end{gather*}
If we sum over $l$, we have
\begin{gather*}
 \sum_{l=-\infty}^{\infty} \frac{\exp \big\{2l \pi {\rm i} w^{-1} \log (-Q) \big\}}{\sin (w^{-1}\big({-}H+\log \alpha_j)\pi + 2l\pi^2{\rm i}w^{-1}\big)}	(-Q)^{w^{-1}(-H+\log \alpha_j)} \\
\qquad{}= - \frac{w(q,q,P\alpha_j Q,q/(P\alpha_j Q);q)_\infty}{\pi(P\alpha_j, q/(P\alpha_j),Q,q/Q;q )_{\infty}},
\end{gather*}
which comes from the following lemma.

\begin{Lemma}[{\cite[equation~(4.3.9), p.~119]{GR04}}]
\begin{gather*}
 \sum_{m=-\infty}^{\infty} \operatorname{csc}\big(\alpha \pi -2 m \pi^2 {\rm i} w^{-1}\big)\operatorname{exp}\big\{ 2m\pi {\rm i} w^{-1} \operatorname{log}(-Q) \big\} (-Q)^{-\alpha}	
= \frac{w (q,q,aQ,q/(aQ);q)_{\infty}}{\pi (a,q/a,Q,q/Q;q)_{\infty}},
\end{gather*}
where $a=q^{\alpha}={\rm e}^{-w\alpha}$.
\end{Lemma}

Summing up the above discussion, we arrive at the following theorem.

\begin{Theorem} \label{Theorem}
Suppose $\alpha_i \notin \alpha_j q^{\mathbb{Z}\backslash \{0\}}$. For $m \geq n$, the $K\big(\mathbb{P}^{n-1}\big)$ valued $q$-series has the following analytic continuation:
\begin{align*}
 P^{l_q(Q)}\sum_{d=0}^{\infty} \frac{\prod_{i=1}^{m}(P{\alpha_i};q)_d}{(Pq;q)_d^n} Q^d
={}& P^{l_q(Q)} \frac{\prod_{i=1}^m(P\alpha_i;q)_\infty}{(Pq;q)_{\infty}^n} \sum_{j=1}^m \frac{(q,q,P\alpha_j Q,q/(P\alpha_j Q);q)_\infty}{(P\alpha_j, q/(P\alpha_j),Q,q/Q;q )_{\infty}}
\\
&\times\frac{\big(\alpha_j^{-1}q;q\big)^n_\infty \cdot {\rm e}^{-1}_{q,\alpha_j}(1/Q)}{\prod_{i=1,i\neq j}^m(\alpha_i/\alpha_j;q)_\infty (q;q)_\infty} W_j(1/Q).
\end{align*}
for $|Q|<1$ and $|\arg (-Q)|<\pi-\delta$, $0<\delta<\pi$.
\end{Theorem}

\begin{Remark}In general, the above formula only contains a part of solutions at $Q=0$.	
\end{Remark}

\section{A special fuchsian case}\label{sec5}
Consider the following difference equation:
\begin{gather}
\Big[\big(1-q^{Q\partial_Q}\big)^4-Q\big(1-q^{\frac{1}{5}}q^{Q\partial_Q}\big) \big(1-q^{\frac{2}{5}}q^{Q\partial_Q}\big)\big(1-q^{\frac{3}{5}}q^{Q\partial_Q}\big)
\big(1-q^{\frac{4}{5}}q^{Q\partial_Q}\big) \Big] F(Q) = 0.
\label{special-fuchsian-case}
\end{gather}
By definition, it is fuchsian. One could easily construct the solutions at $Q=\infty$, indeed, let $w=1/Q$, then (\ref{special-fuchsian-case}) becomes
\begin{gather*}
\Big[ \big(1-q^{-\frac{1}{5}}q^{w\partial_w}\big)\big(1-q^{-\frac{2}{5}}q^{w\partial_w}\big) \big(1-q^{-\frac{3}{5}}q^{w\partial_w}\big)\big(1-q^{-\frac{4}{5}}q^{w\partial_w}\big) - q^2w\big(1-q^{w \partial _w }\big)^4 \Big] G(w)=0.	
\end{gather*}
The characteristic equation of the above difference equation at $w=0$ is
\begin{gather*}
\big(1-q^{-\frac{1}{5}}x\big)\big(1-q^{-\frac{2}{5}}x\big)\big(1-q^{-\frac{3}{5}}x\big) \big(1-q^{-\frac{4}{5}}x\big) = 0,
\end{gather*}
with 4 different roots
\begin{gather*}
 x= q^{\frac{i}{5}}, \qquad i=1,2,3,4.	
\end{gather*}
So the difference equation is non-resonant. By using Frobenius method, we could construct solutions of the form
\begin{gather*}
W_i(w)= e_{q,q^{\frac{i}{5}}}(w) \widetilde{W}_i(w) = e_{q,q^{\frac{i}{5}}}(w) \sum_{n=0}^{\infty} g^i_n w^n .
\end{gather*}
After a short computation one obtain four solutions as follows
\begin{gather}
	W_1(1/Q) = e_{q,q^{\frac{1}{5}}}(1/Q){_4}\phi_3\big(q^{\frac{1}{5}},q^{\frac{1}{5}},q^{\frac{1}{5}},q^{\frac{1}{5}}; q^{\frac{4}{5}}, q^{\frac{3}{5}},q^{\frac{2}{5}};q; q^2/Q \big),
\label{fuchsian-case-infty-1}
\\
	W_2(1/Q) = e_{q,q^{\frac{2}{5}}}(1/Q){_4}\phi_3\big(q^{\frac{2}{5}},q^{\frac{2}{5}},q^{\frac{2}{5}},q^{\frac{2}{5}}; q^{\frac{6}{5}}, q^{\frac{4}{5}},q^{\frac{3}{5}};q; q^2/Q\big),
\label{fuchsian-case-infty-2}
\\
	W_3(1/Q) = e_{q,q^{\frac{3}{5}}}(1/Q){_4}\phi_3\big(q^{\frac{3}{5}},q^{\frac{3}{5}},q^{\frac{3}{5}},q^{\frac{3}{5}}; q^{\frac{7}{5}}, q^{\frac{6}{5}},q^{\frac{4}{5}};q; q^2/Q \big),
\label{fuchsian-case-infty-3}
\\
	W_4(1/Q)= e_{q,q^{\frac{4}{5}}}(1/Q){_4}\phi_3\big(q^{\frac{4}{5}},q^{\frac{4}{5}},q^{\frac{4}{5}},q^{\frac{4}{5}}; q^{\frac{8}{5}}, q^{\frac{7}{5}},q^{\frac{6}{5}};q; q^2/Q \big).
\label{fuchsian-case-infty-4}
\end{gather}

\begin{Remark}
By a small trick, we could find the formula for $\widetilde{W}_i(w)$	 easily. Note that, for $\alpha \notin \mathbb{N}$, we have
\begin{gather*}
\big(1-q^{w \partial_w}\big)^k 	e_{q,q^{\alpha}}(w)\widetilde{W}(w) = e_{q,q^{\alpha}}(w) \big(1-q^\alpha \cdot q^{w \partial_w}\big)^k \widetilde{W}(w).
\end{gather*}
For example, $\widetilde{W}_1(w) $ satisfies the following difference equation
\begin{gather*}
\Big[ \big(1-q^{w\partial_w}\big)\big(1-q^{-\frac{1}{5}}q^{w\partial_w}\big) \big(1-q^{-\frac{2}{5}}q^{w\partial_w}\big)\big(1-q^{-\frac{3}{5}}q^{w\partial_w}\big) - q^2w\big(1-q^{\frac{1}{5}}q^{w \partial _w }\big)^4 \Big] \widetilde{W}_1(w) =0.
\end{gather*}
From the above explicit form, it's quite easy to find a $q$-hypergeometric series representation for~$\widetilde{W}_1(w)$.
\end{Remark}

Let's consider the solutions of (\ref{special-fuchsian-case}) at $Q=0$, the characteristic equation is as follows
\begin{gather*}
(1-x)^4=0.	
\end{gather*}
From the general theorem (see \cite{Adams1931, RW21} for details), we have solutions of the form
\begin{gather}
G_i(Q)=l_q(Q)G_{i-1}(Q) + g_i(Q), \qquad i=2,3,4	 \label{second-solution-Q=0},
\end{gather}
where
\begin{gather}\label{first-solution-Q=0}
G_1(Q)= \sum_{d=0}^{\infty} \frac{\prod_{i=1}^4\big(q^{\frac{i}{5}};q\big)_d}{(q;q)^4_d}	 Q^d = {_4}\phi_3\big(q^{\frac{1}{5}},q^{\frac{2}{5}},q^{\frac{3}{5}},q^{\frac{4}{5}}; q, q,q;q; Q \big)
\end{gather}
is a solution of (\ref{special-fuchsian-case}) and $g_i(Q)$ are power series.

For general $q$-hypergeometric function $ {_4}\phi _3 (a_1,a_2,a_3,a_4;b_1,b_2,b_3;q;z) $, we have the following famous transformation formula.

\begin{Proposition}[{\cite[equation~(4.5.2), p.~120]{GR04}}]
\begin{gather}
{_4}\phi _3 (a_1,a_2,a_3,a_4;b_1,b_2,b_3;q;z)\nonumber
\\ \qquad
{}= \frac{(a_2,a_3,a_4,b_1/a_1,b_2/a_1,b_3/a_1,a_1z,q/a_1z;q)_\infty} {(b_1,b_2,b_3,a_2/a_1,a_3/a_1,a_4/a_1,z,q/z;q )_\infty}\nonumber
\\ \qquad\hphantom{=}
 {}\times {_4}\phi_3\bigg(a_1,a_1q/b_1,a_1q/b_2,a_1q/b_3; a_1q/a_2,a_1q/a_3,a_1q/a_4;q; \frac{qb_1b_2b_3}{za_1a_2a_3a_4}\bigg) \nonumber
 \\ \qquad\hphantom{=}
{}+ \operatorname{idem}(a_1,a_2,a_3,a_4). \label{4-phi-3-indentity}
\end{gather}
The symbol `` $\operatorname{idem}(a_1,a_2,a_3,a_4)$'' after an expression stands for the sum of the $3$ expressions obtained from the preceding expression by interchanging $a_1$ with $a_k$, $k=2,3,4$.
\end{Proposition}
So in our special case, (\ref{first-solution-Q=0}) can be written as a combination of (\ref{fuchsian-case-infty-1})--(\ref{fuchsian-case-infty-4}). As mentioned before, the other solutions has the form of (\ref{second-solution-Q=0}) which is hard to compute. Thus, it's very hard to find the connection matrix.

\subsection{Connection matrix}\label{section5.1}
Notice that (\ref{special-fuchsian-case}) is a special case of (\ref{aux-dif-F_{m,n}}) with $n=m=4$ and
\begin{gather*}
\alpha_i = q^{\frac{i}{5}}, \qquad i=1,2,3,4,	
\end{gather*}
then
\begin{gather}
F_{4,4}\big(\big\{q^{\frac{i}{5}}\big\}_{i=1}^4,Q\big) = P^{l_q(Q)}\sum_{d=0}^{\infty} \frac{\prod_{i=1}^{4}\big(Pq^{\frac{i}{5}};q\big)_d}{(Pq;q)_d^4} Q^d. \label{auxiliary-q-series}
\end{gather}
The solutions of difference equation (\ref{special-fuchsian-case}) at $Q=0$ are given by the expansion of $F_{4,4}\big(\big\{q^{\frac{i}{5}}\big\}_{i=1}^4,Q\big)$ with respect to $(1-P)^i$, $i=0,1,2,3$. From Theorem \ref{Theorem}, then we have

\begin{Corollary} 
\begin{align}
 P^{l_q(Q)}\sum_{d=0}^{\infty} \frac{\prod_{i=1}^{4}\big(Pq^{\frac{i}{5}};q\big)_d}{(Pq;q)_d^4} Q^d
{}&= P^{l_q(Q)} \frac{\prod_{i=1}^4\big(Pq^{\frac{i}{5}};q\big)_\infty}{(Pq;q)_{\infty}^4} \sum_{j=1}^4 \frac{\big(q,q,Pq^{\frac{j}{5}} Q,q/\big(Pq^{\frac{j}{5}} Q\big);q\big)_\infty}\nonumber
\\
&\times{\big(Pq^{\frac{j}{5}}, q/\big(Pq^{\frac{j}{5}}\big),Q,q/Q;q \big)_{\infty}} \frac{\big(q^{-\frac{j}{5}}q;q\big)^4_\infty {\rm e}^{-1}_{q,q^{j/5}}(1/Q) }{\prod_{i=1,i\neq j}^4\big(q^{\frac{i-j}{5}};q\big)_\infty (q;q)_\infty }\nonumber
\\
&\times W_j(1/Q). 	\label{Corollary-2-formula}
\end{align}
\end{Corollary}

\begin{Remark}
Taking $P=1$ in the above formula, we obtain
\begin{align*}
\sum_{d=0}^{\infty} \frac{\prod_{i=1}^{4}\big(q^{\frac{i}{5}};q\big)_d}{(q;q)_d^4} Q^d
={}&\sum_{k=1}^{4} \frac{\big(q^{\frac{1}{5}} \cdots \hat{k} \cdots q^{\frac{4}{5}},q^{1-\frac{k}{5}},q^{1-\frac{k}{5}},q^{1-\frac{k}{5}};q\big)_{\infty}} {\big(q^{\frac{1-k}{5}},\dots \hat{k} \dots, q^{\frac{4-k}{5}},q,q,q;q \big)_{\infty}} \frac{\big(q^{\frac{k}{5}}Q;q^{\frac{5-k}{5}}/Q\big)_\infty}{(Q,q/Q)_\infty}
\\
& \times {_4}\phi_3\big(q^{\frac{k}{5}},q^{\frac{k}{5}},q^{\frac{k}{5}},q^{\frac{k-1}{5}+1}; q^{\frac{k-2}{5}+1},\dots \hat{k} \dots, q^{\frac{k-4}{5}+1};q; q^2/Q\big).
\end{align*}	
It agrees with (\ref{4-phi-3-indentity}).
\end{Remark}

{\sloppy In order to obtain the connection matrix, we need to expand (\ref{Corollary-2-formula}) with respect to \mbox{$\{ (1-P)^k \}^3_{k=1}$}. Notice that
\begin{gather*}
P^{l_q(Q)} = (1-(1-P))^{l_q(Q)} = \sum_{k \geq 0} (-1)^k \binom{\ell_{q}(Q)}{k}	(1-P)^k,
\end{gather*}}
where
\begin{gather*}
	\binom{\ell_{q}(Q)}{k}=\frac{1}{k !} \prod_{r=0}^{k-1}(\ell_{q}(Q)-r).
\end{gather*}
Then
\begin{gather*}
P^{l_q(Q)}\sum_{d=0}^{\infty} \frac{\prod_{i=1}^{4}\big(Pq^{\frac{i}{5}};q\big)_d}{(Pq;q)_d^4} Q^d = \sum_{m = 0}^3 \sum_{a+b=m} (-1)^a \binom{\ell_{q}(Q)}{a} X_b(q,Q) (1-P)^m,	 
\end{gather*}
where $X_b(q,Q)$ is the coefficient of
\begin{gather}
\sum_{d=0}^{\infty} \frac{\prod_{i=1}^{4}\big(Pq^{\frac{i}{5}};q\big)_d}{(Pq;q)_d^4} Q^d	
= \sum_{b = 0}^3 X_b(q,Q) (1-P)^b. \label{X_b}
\end{gather}

Let's consider the expansion of $(Pq^{\alpha};q)_\infty $ and $\big(P^{-1}q^{\alpha};q\big)_\infty $. By definition
\begin{gather*}
(Pq^{\alpha};q)_\infty = \prod_{k=0}	^{\infty} \big(1-Pq^{\alpha+k}\big).
\end{gather*}
For $(1-Pq^{\alpha+k})$, we have
\begin{gather*}
	\big(1-Pq^{\alpha+k}\big)= \big(1-q^{\alpha+k}+q^{\alpha+k}(1-P)\big)=\big(1-q^{\alpha+k}\big)\bigg(1+ \frac{q^{\alpha+k}}{1-q^{\alpha+k}}(1-P) \bigg).
\end{gather*}
Then
\begin{gather*}
\prod_{k=0}^{\infty} \big(1-q^{\alpha+k}\big)\bigg(1+ \frac{q^{\alpha+k}}{1-q^{\alpha+k}}(1-P) \bigg)
\\ \qquad
{}= (q^\alpha;q)_{\infty} \prod_{k=0}^{\infty} \Bigg[ 1 + \sum_{k=0}^{\infty}\frac{q^{\alpha+k}}{1-q^{\alpha+k}} (1-P) + \sum_{i<j}^{\infty} \frac{q^{2\alpha +i+j}}{(1-q^{\alpha+i})(1-q^{\alpha+j})} (1-P)^2
\\ \qquad\hphantom{=}
{}+ \sum_{i<j<l}^{\infty}\frac{q^{3\alpha +i+j+l}}{(1-q^{\alpha+i})(1-q^{\alpha+j})(1-q^{\alpha+l})} (1-P)^3 + O\big((1-P)^4\big) \Bigg].	
\end{gather*}
Similarly,
\begin{align*}
\big(1-P^{-1}q^{\alpha+k}\big) ={}& 1-\frac{q^{\alpha+k}}{1-(1-P)}
=\big(1-q^{\alpha+k}\big)-q^{\alpha+k}(1-P) - q^{\alpha+k}(1-P)^2
\\
&-q^{\alpha+k}(1-P)^3 + O\big((1-P)^4\big).	
\end{align*}
 Then
\begin{gather*}
\prod_{k=0}^{\infty}\big(\big(1-q^{\alpha+k}\big)-q^{\alpha+k}(1-P)- q^{\alpha+k}(1-P)^2-q^{\alpha+k}(1-P)^3 + O\big((1-P)^4\big) \big)
\\ \qquad
{}= (q^\alpha;q)_{\infty} \prod_{k=0}^{\infty} \Bigg[ 1 - \sum_{k=0}^{\infty}\frac{q^{\alpha+k}}{1-q^{\alpha+k}} (1-P)
\\ \qquad\hphantom{=}
{}+ \Bigg( \sum_{i<j}^{\infty} \frac{q^{2\alpha +i+j}}{(1-q^{\alpha+i})(1-q^{\alpha+j})} -\sum_{k=0}^{\infty}\frac{q^{\alpha+k}}{1-q^{\alpha+k}} \Bigg) (1-P)^2
\\ \qquad\hphantom{=}
{}+ \Bigg({-}\sum_{i<j<l}^{\infty}\frac{q^{3\alpha +i+j+l}}{(1-q^{\alpha+i})(1-q^{\alpha+j})(1-q^{\alpha+l})} +2\sum_{i<j}^{\infty} \frac{q^{2\alpha +i+j}}{(1-q^{\alpha+i})(1-q^{\alpha+j})}
\\ \qquad\hphantom{=}
{}-\sum_{k=0}^{\infty}\frac{q^{\alpha+k}}{1-q^{\alpha+k}} \Bigg)(1-P)^3 \Bigg] + O\big((1-P)^4\big)	.
\end{gather*}

In order to simplify the computation, we introduce the following notations
\begin{gather*}
f_1(x) = \sum_{k=0}^{\infty}\frac{xq^{k}}{1-xq^{k}}, \\
f_2(x) =	\sum_{i<j}^{\infty} \frac{x^2q^{i+j}}{(1-xq^{i})(1-xq^{j})}, \\
f_3(x) = \sum_{i<j<l}^{\infty}\frac{x^3q^{i+j+l}}{(1-xq^{i})(1-xq^{j})(1-xq^{l})}
\end{gather*}
and
\begin{gather*}
	F_1(x_1,x_2,x_3,x_4) = -\sum_{k=1}^4 f_1(x_k), \\
	F_2(x_1,x_2,x_3,x_4) = -\sum_{k=1}^4 f_2(x_k) + \sum_{i<j}^4f_1(x_i)f_1(x_j), \\
	F_3(x_1,x_2,x_3,x_4) = -\sum_{k=1}^4 f_3(x_k) + \sum_{i<j}^4 (f_1(x_i)f_2(x_j)+ f_2(x_i)f_1(x_j))
\\ \hphantom{F_3(x_1,x_2,x_3,x_4) = }
{}- \sum_{i<j<l}^4f_1(x_i)f_1(x_j)f_1(x_l).
\end{gather*}
With a little computation, we obtain
\begin{align}
\frac{\prod_{i=1}^4\big(Pq^{\frac{i}{5}};q\big)_{\infty} }{(Pq;q)^4_\infty} = {}& \frac{\prod_{k=1}^{4}\big(q^{\frac{k}{5}};q\big)_{\infty}}{(q;q)^4_{\infty}}
\big[ 1 +\big(F_1\big(q^{\frac{\bullet}{5}}\big) -F_{1}(q)\big)(1-P) \nonumber
\\
&+ \big(F_1(q)^2-F_1\big(q^{\frac{\bullet}{5}}\big)F_1(q) - F_2(q) +F_2\big(q^{\frac{\bullet}{5}}\big) \big)(1-P)^2 \nonumber
\\
&+ \big( F_1(q)^3+F_1\big(q^{\frac{\bullet}{5}}\big) \big(F_1(q)^2-F_2(q) \big) +2F_1(q)F_2(q)-F_1(q)F_2\big(q^{\frac{\bullet}{5}}\big)\nonumber
\\
&- F_3(q) +F_3\big(q^{\frac{\bullet}{5}}\big) \big)(1-P)^3 \big]
	+ O\big((1-P)^4\big) \label{F_k} .
\end{align}
Here we use the notations:
\begin{gather*}
F_i\big(q^{\frac{\bullet}{5}}\big) := F_{i}\big(q^{\frac{1}{5}},q^{\frac{2}{5}},q^{\frac{3}{5}},q^{\frac{4}{5}}\big), 	
\\
F_{i}(q) := F_{i}(q,q,q,q).
\end{gather*}
{\samepage
For simplicity, we write the above formula as
\begin{gather}
\frac{\prod_{i=1}^4\big(Pq^{\frac{i}{5}};q\big)_{\infty} }{(Pq;q)^4_\infty} = 	\frac{\prod_{k=1}^{4}\big(q^{\frac{k}{5}};q\big)_{\infty}}{(q;q)^4_{\infty}} \Bigg[ 1+ \sum_{k=1}^{3}\mathbb{F}_k \cdot (1-p)^k + O\big((1-P)^4\big) \Bigg], \label{mathbb{F}}
\end{gather}
where $\mathbb{F}_k$ stands for the coefficient of $(1-P)^k$ in (\ref{F_k}).}

Similarly, we consider
\begin{gather*}
(Px;q)_{\infty}\big(P^{-1}x^{-1}q;q\big)_{\infty}.
\end{gather*}
Then
\begin{gather*}
\big(1-Pxq^d\big)\big(1-qP^{-1}x^{-1}q^d\big)
\\ \qquad
{}= \big(1-xq^d+xq^d(1-P) \big)
\big( \big(1-qx^{-1}q^d\big)-qx^{-1}q^d(1-P)-qx^{-1}q^d(1-P)^2
\\ \qquad \hphantom{=}
{}-qx^{-1}q^d(1-P)^3 +O\big((1-P)^4\big) \big)
\\ \qquad
{}= \big(1-xq^d\big)\big(1-qx^{-1}q^d\big) \bigg[ 1+ \frac{xq^d\big(1-qx^{-2}\big)}{\big(1-xq^d\big)\big(1-qx^{-1}q^d\big)}(1-P)
\\ \qquad \hphantom{=}
- \frac{q^{d+1}x^{-1}}{(1\!-xq^d)(1\!-qx^{-1}q^d)}(1-P)^2 -\frac{q^{d+1}x^{-1}}{(1\!-xq^d)(1\!-qx^{-1}q^d)}(1-P)^3\! +O\big((1-P)^4\big) \bigg].
\end{gather*}
Set
\begin{gather*}
g_1(x) = \sum_{d=0}^\infty \frac{xq^d\big(1-qx^{-2}\big)}{(1-xq^d)\big(1-qx^{-1}q^d\big)},
\\
g_2(x) = \sum_{i<j}^{\infty} \prod_{k=i,j}\bigg( \frac{xq^k\big(1-qx^{-2}\big)}{(1-xq^k)\big(1-qx^{-1}q^k\big)} \bigg) -\sum_{d=0}^{\infty} \frac{q^{d+1}x^{-1}}{(1-xq^d)\big(1-qx^{-1}q^d\big)},
\\
g_3(x) = \sum_{i<j<l}^{\infty} \prod_{k=i,j,l}\bigg( \frac{xq^k\big(1-qx^{-2}\big)}{(1-xq^k)\big(1-qx^{-1}q^k\big)} \bigg) -\sum_{i \neq j} \frac{q^{i+j+1}\big(1-qx^{-2}\big)}{\prod_{k=i,j}(1-xq^k)\big(1-qx^{-1}q^k\big)}
\\ \hphantom{g_3(x) =}
{}- \sum_{d=0}^\infty \frac{q^{d+1}x^{-1}}{(1-xq^d)\big(1-qx^{-1}q^d\big)}.
\end{gather*}
Then
\begin{gather*}
(Px;q)_{\infty}\big(P^{-1}x^{-1}q;q\big)_{\infty}
 \\ \qquad
{}= \frac{\theta_q(-x)}{(q;q)_{\infty}}\big( 1 + g_1(x)(1-P) + g_2(x)(1-P)^2 + g_3(x)(1-P)^3 +O\big((1-P)^4\big) \big).	
\end{gather*}
So we obtain
\begin{align}
\frac{\big(Pq^{\frac{k}{5}}Q,P^{-1}Q^{-1}q^{1-\frac{k}{5}};q\big)_{\infty}} {\big(Pq^{\frac{k}{5}},P^{-1}q^{1-\frac{k}{5}};q\big)_{\infty}} 	
= {}&\frac{\theta_q\big({-}q^{\frac{k}{5}}Q\big)}{\theta_q\big({-}q^{\frac{k}{5}}\big)} \big[ 1+ \big( g_1\big(q^{\frac{k}{5}}Q\big)-g_1\big(q^{\frac{k}{5}}\big) \big)(1-P) \nonumber
\\
&+\! \big({-}g_1\big(q^{\frac{k}{5}}Q\big)g_1\big(q^{\frac{k}{5}}\big)\! +\!g_1^2\big(q^{\frac{k}{5}}\big)\!+\!g_2\big(q^{\frac{k}{5}}Q\big) \!-\!g_2\big(q^{\frac{k}{5}}\big) \big)(1\!-\!P)^2 \nonumber
\\
&+\! \big({-}g_1^3\big(q^{\frac{k}{5}}\big) \!-\!g_1\big(q^{\frac{k}{5}}\big)g_2\big(q^{\frac{k}{5}}Q\big)\!+\!g_1\big(q^{\frac{k}{5}}Q\big) \big( g^2_1\big(q^{\frac{k}{5}}\big) \!-\!g_2\big(q^{\frac{k}{5}}\big) \big)\nonumber
\\
&+ 2g_1\big(q^{\frac{k}{5}}\big)g_2\big(q^{\frac{k}{5}}\big) +g_3\big(q^{\frac{k}{5}}Q\big)-g_3\big(q^{\frac{k}{5}}\big) \big)(1-P)^3 \big] \nonumber
\\
&+ O\big((1-P)^4\big). \label{G_k}
\end{align}
For simplicity, we write the above formula as follows
\begin{gather}
\frac{\big(Pq^{\frac{k}{5}}Q,P^{-1}Q^{-1}q^{1-\frac{k}{5}};q\big)_{\infty}} {\big(Pq^{\frac{k}{5}},P^{-1}q^{1-\frac{k}{5}};q\big)_{\infty}} = \frac{\theta_q\big({-}q^{\frac{k}{5}}Q\big)}{\theta_q\big({-}q^{\frac{k}{5}}\big)}
\Bigg[1 + \sum_{k=1}^3 \mathbb{G}_k \cdot (1-P)^k + O\big((1-P)^4\big) \Bigg],	 \label{mathbb{G}}
\end{gather}
where $\mathbb{G}_k $ stands for the coefficient of $(1-P)^k$ in (\ref{G_k}).

In conclusion, we arrive at the following corollary.
\begin{Corollary}
The $4$ solutions of \eqref{special-fuchsian-case} at $Q=0$ are given as the expansion of \eqref{auxiliary-q-series}, i.e.
\begin{align*}
\sum_{a+b=m} (-1)^a \binom{\ell_{q}(Q)}{a} X_b(q,Q), \qquad m=0,1,2,3	,
\end{align*}
where $X_b(q,Q)$ is defined in \eqref{X_b}. The $4$ solutions of \eqref{special-fuchsian-case} at $Q=\infty$ are given explicitly as \eqref{fuchsian-case-infty-1}--\eqref{fuchsian-case-infty-4}. The connection matrix is as follows
\begin{gather*}
\bullet\quad
X_0(q,Q)=\sum_{d=0}^{\infty} \frac{\prod_{i=1}^{4}\big(q^{\frac{i}{5}};q\big)_d}{(q;q)_d^4} Q^d
\\ \hphantom{\bullet\quad X_0(q,Q)}
{}=\sum_{k=1}^{4} \frac{\big(q^{\frac{1}{5}} \cdots \hat{k} \cdots q^{\frac{4}{5}},q^{1-\frac{k}{5}},q^{1-\frac{k}{5}},q^{1-\frac{k}{5}};q\big)_{\infty}} {\big(q^{\frac{1-k}{5}},\dots \hat{k} \dots, q^{\frac{4-k}{5}},q,q,q;q\big)_{\infty}} \frac{\big(q^{\frac{k}{5}}Q;q^{\frac{5-k}{5}}/Q\big)_\infty}{(Q,q/Q)_\infty}
\\ \hphantom{\bullet\quad X_0(q,Q)=}
{}\times {\rm e}^{-1}_{q,q^{k/5}}(1/Q) W_k(1/Q).
\\
\bullet\quad
X_1(q,Q)	 = \sum_{k=1}^{4} \frac{\big(q^{\frac{1}{5}} \cdots \hat{k} \cdots q^{\frac{4}{5}},q^{1-\frac{k}{5}},q^{1-\frac{k}{5}},q^{1-\frac{k}{5}};q\big)_{\infty}} {\big(q^{\frac{1-k}{5}},\dots \hat{k} \dots, q^{\frac{4-k}{5}},q,q,q;q \big)_{\infty}} \frac{\big(q^{\frac{k}{5}}Q;q^{\frac{5-k}{5}}/Q\big)_\infty}{(Q,q/Q)_\infty}
\\ \hphantom{\bullet\quad X_1(q,Q)	 =}
{}\times (\mathbb{G}_1+ \mathbb{F}_1) {\rm e}^{-1}_{q,q^{k/5}}(1/Q) W_k(1/Q).
\\
\bullet\quad
X_2(q,Q)	 = \sum_{k=1}^{4} \frac{\big(q^{\frac{1}{5}} \cdots \hat{k} \cdots q^{\frac{4}{5}},q^{1-\frac{k}{5}},q^{1-\frac{k}{5}},q^{1-\frac{k}{5}};q\big)_{\infty}} {\big(q^{\frac{1-k}{5}},\dots \hat{k} \dots, q^{\frac{4-k}{5}},q,q,q;q \big)_{\infty}} \frac{\big(q^{\frac{k}{5}}Q;q^{\frac{5-k}{5}}/Q\big)_\infty}{(Q,q/Q)_\infty}
\\ \hphantom{\bullet\quad X_2(q,Q)	 =}
{}\times (\mathbb{G}_2+ \mathbb{F}_2+ \mathbb{G}_1 \mathbb{F}_1) {\rm e}^{-1}_{q,q^{k/5}}(1/Q) W_k(1/Q).
\\
\bullet\quad
X_3(q,Q)	 = \sum_{k=1}^{4} \frac{\big(q^{\frac{1}{5}} \cdots \hat{k} \cdots q^{\frac{4}{5}},q^{1-\frac{k}{5}},q^{1-\frac{k}{5}},q^{1-\frac{k}{5}};q\big)_{\infty}} {\big(q^{\frac{1-k}{5}},\dots \hat{k} \dots, q^{\frac{4-k}{5}},q,q,q;q \big)_{\infty}} \frac{\big(q^{\frac{k}{5}}Q;q^{\frac{5-k}{5}}/Q\big)_\infty}{(Q,q/Q)_\infty}
\\ \hphantom{\bullet\quad X_3(q,Q)	 =}
{}\times(\mathbb{G}_3+ \mathbb{F}_3+ \mathbb{G}_2 \mathbb{F}_1 + \mathbb{G}_1 \mathbb{F}_2 ) {\rm e}^{-1}_{q,q^{k/5}}(1/Q) W_k(1/Q).
\end{gather*}
Here, for simplicity, we use the notations $\mathbb{G}_k$ and $\mathbb{F}_k$ defined in \eqref{mathbb{F}} and \eqref{mathbb{G}}.
\end{Corollary}

\section[Confluence of the q-difference structure]
{Confluence of the $\boldsymbol q$-difference structure}\label{sec6}
Notice that
\begin{gather*}
\lim_{q \rightarrow 1} \frac{1-q^{Q\partial_Q}}{1-q} = Q\frac{\rm d}{{\rm d} Q},	
\end{gather*}
then one could easily see that the following difference equation is confluent to (\ref{differential-equation}), i.e.,
\begin{gather*}
\lim_{q \rightarrow 1} 	\Big[ \big(1-q^{Q\partial_Q}\big)^4-Q\big(1-q^{\frac{1}{5}}q^{Q\partial_Q}\big)
\big(1-q^{\frac{2}{5}}q^{Q\partial_Q}\big)\big(1-q^{\frac{3}{5}}q^{Q\partial_Q}\big) \big(1-q^{\frac{4}{5}}q^{Q\partial_Q}\big) \Big]/(1-q)^4
\\ \qquad
{}= \bigg[ \bigg(Q\frac{\rm d}{{\rm d}Q}\bigg)^4- Q\bigg(Q\frac{\rm d}{{\rm d}Q}+\frac{1}{5}\bigg)\bigg(Q\frac{\rm d}{{\rm d}Q}+\frac{2}{5}\bigg)\bigg(Q\frac{\rm d}{{\rm d}Q}+\frac{3}{5}\bigg)\bigg(Q\frac{\rm d}{{\rm d}Q}+\frac{4}{5}\bigg) \Bigg].
\end{gather*}
In the following, we set $q(t)={\rm e}^{-t}$ and $P = q^H = {\rm e}^{-tH}$.

\begin{Lemma}
\begin{gather*}
\lim_{t \rightarrow 0} 	P^{l_q(Q)}\sum_{d=0}^{\infty} \frac{\prod_{i=1}^{4}\big(Pq^{\frac{i}{5}};q\big)_d}{(Pq;q)_d^4} Q^d = Q^{H} \sum_{d=0}^{\infty} \frac{\prod_{i=1}^{4}\big(H+\frac{i}{5}\big)_d}{(H+1)_d^4} Q^d \mod\big(H^4\big).
\end{gather*}

\begin{proof}
Since
\begin{gather*}
\lim_{t \rightarrow 0} \frac{\prod_{k=1}^d\big(1-Pq^k\big)}{\prod_{k=1}^d\big(1-q^k\big)} = \frac{\prod_{k=1}^d(H+k)}{d!}, \qquad
\lim_{t \rightarrow 0} \frac{(1-P)^k}{(1-q)^k} = H^k,
\end{gather*}
and	
\begin{gather*}
P^{l_q(Q)} = \sum_{k \geq 0} (-1)^k \binom{\ell_{q}(Q)}{k}	(1-P)^k	 = \sum_{k \geq 0} (q-1)^k \binom{\ell_{q}(Q)}{k}\frac{(1-P)^k}{(1-q)^k}.
\end{gather*}
Then from Proposition \ref{confluence-ell-e}, we arrive at the conclusion.
\end{proof}

\end{Lemma}
The $q$-Gamma function is defined as follows
\begin{gather*}
	 \Gamma_q(x) = \frac{(q;q)_\infty}{(q^x;q)_\infty}(1-q)^{1-x}.	
\end{gather*}
It has a nice property
\begin{gather*}
\lim_{q \rightarrow 1} \Gamma_q(x) = \Gamma(x).	
\end{gather*}
Using $q$-Gamma function, we rewrite (\ref{Corollary-2-formula}) in the following form
\begin{gather}
P^{l_q(Q)}\sum_{d=0}^{\infty} \frac{\prod_{i=1}^{4}\big(Pq^{\frac{i}{5}};q\big)_d}{(Pq;q)_d^4} Q^d \nonumber
\\ \qquad
{}= P^{l_q(Q)}\frac{\Gamma_q(H+1)^4}{\Gamma_q\big(H+\frac{1}{5}\big)\Gamma_q\big(H+\frac{2}{5}\big)
\Gamma_q\big(H+\frac{3}{5}\big)\Gamma_q\big(H+\frac{4}{5}\big)} \sum_{k=1}^{4} \frac{\Gamma_q\big(\frac{1-k}{5}\big) \cdots \hat{k} \cdots \Gamma_q\big(\frac{4-k}{5}\big)}{\Gamma_q\big(1-\frac{k}{5}\big)^4}\nonumber
\\ \qquad\hphantom{=}
{}\times \Gamma_q\bigg(H+\frac{k}{5}\bigg) \Gamma_q\bigg({-}H+1-\frac{k}{5}\bigg) \frac{\theta_{q}\big({-}q^{H+k/5}Q\big)}{\theta_q(-Q)} {\rm e}^{-1}_{q,q^{k/5}}(1/Q) W_k(1/Q).	 \label{rewrite-in-q-Gamma}
\end{gather}
After taking limit, we arrive at the following proposition.

\begin{Proposition}
\begin{gather}
Q^{H}\sum_{d \geq 0} \frac{\prod_{k=1}^{5d}(5H+k)}{\prod_{k=1}^d(H+k)^5} \big(Q/5^5\big)^{d} \nonumber
\\ \qquad
{} = \frac{5^{5H} \Gamma(H+1)^5 }{\Gamma(5H+1)} \sum_{k=1}^{4} \frac{5^{k-1} \Gamma(5-k)}{\prod_{i=1,i\neq k}^4(i-k) \Gamma\big(1-\frac{k}{5}\big)^5} \frac{\pi {\rm e}^{-\pi {\rm i} \left(H + \frac{k}{5}\right)} }{\sin\big(\pi\big(H+\frac{k}{5}\big)\big) } \widetilde{W}_k(1/Q), \label{Prop6}
 \end{gather}
where
\begin{gather*}
\widetilde{W}_k=\sum_{d \geq 0} \frac{\prod_{l=0}^{d-1}\big(\frac{k}{5}+l\big)^5}{\prod_{l=0}^{5d-1}(k+l)} \big(5^5/Q\big)^d.
\end{gather*}
\end{Proposition}

\begin{proof}
After taking limit, the left hand side of (\ref{rewrite-in-q-Gamma}) becomes	
\begin{gather*}
Q^{H} \sum_{d=0}^{\infty} \frac{\prod_{i=1}^{4}\big(H+\frac{i}{5}\big)_d}{(H+1)_d^4}Q^d	= Q^H \sum_{d=0}^{\infty} \prod_{i=1}^{4} \bigg( \frac{\Gamma\big(H+\frac{i}{5}+d\big)}{\Gamma\big(H+\frac{i}{5}\big)} \bigg) \bigg(\frac{\Gamma(H+1)}{\Gamma(H+d+1)} \bigg)^4 Q^d.
\end{gather*}
Recall some formulas of Gamma function:
\begin{gather*}
\Gamma(x)\Gamma(1-x) = \frac{\pi}{\sin(\pi x)},
\\
\Gamma(nx)(2\pi)^{(n-1)/2} = n^{nx-\frac{1}{2}} \Gamma(x) \Gamma\bigg(x+\frac{1}{n}\bigg) \cdots \Gamma\bigg(x+\frac{n-1}{n}\bigg).	
\end{gather*}
Then
\begin{gather*}
\prod_{i=1}^{4} \bigg(\frac{\Gamma\big(H+\frac{i}{5}+d\big)}{\Gamma\big(H+\frac{i}{5}\big)}\bigg) \frac{\Gamma(H+d+1)}{\Gamma(H+1)} = \frac{1}{5^{5d}} \frac{\Gamma(5H+5d+1)}{\Gamma(5H+1)}.	
\end{gather*}
Thus, we arrive at the left-hand side of (\ref{Prop6}).

After taking limit, the right-hand side of (\ref{rewrite-in-q-Gamma}) becomes
\begin{gather*}
Q^H \frac{\Gamma(H+1)^4}{\Gamma\big(H+\frac{1}{5}\big) \cdots \Gamma\big(H+\frac{4}{5}\big)} \sum_{k=1}^{4} \frac{\Gamma\big(\frac{1-k}{5}\big) \cdots \hat{k} \cdots \Gamma\big(\frac{4-k}{5}\big) }{\Gamma\big(1-\frac{k}{5}\big)^4}	
 \Gamma\bigg(H+\frac{k}{5}\bigg)\Gamma\bigg({-}H+1-\frac{k}{5}\bigg) \\ \qquad
 {}\times (-Q)^{-H-\frac{k}{5}} \widetilde{W}_k(1/Q)
= \frac{\Gamma(H+1)^4}{\Gamma\big(H+\frac{1}{5}\big) \cdots \Gamma\big(H+\frac{4}{5}\big)} \sum_{k=1}^{4} \frac{\Gamma\big(\frac{1-k}{5}\big) \cdots \hat{k} \cdots \Gamma\big(\frac{4-k}{5}\big) }{\Gamma\big(1-\frac{k}{5}\big)^4}	
\\ \qquad
{}\times \Gamma\bigg(H+\frac{k}{5}\bigg)\Gamma\bigg({-}H+1-\frac{k}{5}\bigg) {\rm e}^{ \pi {\rm i} \left(-H -\frac{k}{5}\right)} \widetilde{W}_k(1/Q),
\end{gather*}
similarly,
\begin{gather*}
\frac{\Gamma(H+1)^4}{\Gamma\big(H+\frac{1}{5}\big) \cdots \Gamma\big(H+\frac{4}{5}\big)}
= \frac{ \Gamma(H+1)^5 }{\Gamma(5H+1)} \frac{5^{5H+1/2}}{(2\pi)^{2}},
\end{gather*}
and
\begin{align*}
\frac{\Gamma\big(\frac{1-k}{5}\big) \cdots \hat{k} \cdots \Gamma\big(\frac{4-k}{5}\big) }{\Gamma\big(1-\frac{k}{5}\big)^4} &=\frac{5^3}{(1-k)\cdots \hat{k} \cdots (4-k)} \frac{\prod_{i=1}^4\Gamma\big(1-\frac{k}{5}+\frac{i}{5}\big) \Gamma\big(1-\frac{k}{5}\big)}{\Gamma\big(1-\frac{k}{5}\big)^5}
\\
&=\frac{5^{k-1-1/2}(2\pi)^2}{(1-k)\cdots \hat{k} \cdots (4-k)} \frac{ \Gamma(5-k)}{\Gamma\big(1-\frac{k}{5}\big)^5}.
\end{align*}
Thus we arrive at the right-hand side of (\ref{Prop6}).
\end{proof}

\begin{Remark}
Recall that in the introduction, we have the change of variables
\begin{gather*}
Q=5^5{\rm e}^t.	
\end{gather*}
Under the above change of variables, (\ref{Prop6}) becomes
\begin{gather*}
{\rm e}^{tH}\sum_{d \geq 0} \frac{\prod_{k=1}^{5d}(5H+k)}{\prod_{k=1}^d(H+k)^5} {\rm e}^{td} = \frac{ \Gamma(H+1)^5 }{\Gamma(5H+1)} \sum_{k=1}^{4} \frac{5^{k-1} \Gamma(5-k)}{\prod_{i=1,i\neq k}^4(i-k) \Gamma\big(1-\frac{k}{5}\big)^5} \frac{\pi {\rm e}^{-\pi {\rm i} \left(H + \frac{k}{5}\right)} }{\sin\big(\pi\big(H+\frac{k}{5}\big)\big) } \widetilde{W}_k.
\end{gather*}
From \cite{CR13}, we know
\begin{gather*}
	\frac{ \Gamma(H+1)^5 }{\Gamma(5H+1)} = 1 + C (2\pi {\rm i})^2 H^2 - E(2 \pi {\rm i})^3 H^3 + O\big(H^4\big),
\end{gather*}
where $C=5/12$ and $E=- \xi(3)40/(2\pi {\rm i})^3 $ with $\xi(3)$ equals to Ap\'ery's constant, i.e., it is related to the intersection theory of the quintic three-fold. We hope the expansion of the above equation on both sides with respect to the basis $\{H^i \}_{i=0}^3$ will match the result in \cite[formula~(53)]{CR13} up to the monodromy at $0 $ and $\infty$. For additional discussion on confluence, see~\cite{Ro19} for projective spaces, and~\cite{MR21} for weak Fano manifolds.
\end{Remark}

\subsection*{Acknowledgements}
The author would like to thank Professor Yongbin Ruan for suggesting this problem and for valuable discussions. Thanks are also due to Professor Shuai Guo and Dr.\ Yizhen Zhao for their helpful discussion. This work was initiated during the author's stay at the Institute For Advanced Study In Mathematics (IASM) at Zhejiang University. The author would like to express his thanks to IASM, Professor Bohan Fan, Professor Huijun Fan, and Peking University for their helpful support during this visit. The author wants to thank the anonymous referees who help improve the paper a lot. The author is supported by a KIAS Individual Grant (MG083901) at Korea Institute for Advanced Study.

\pdfbookmark[1]{References}{ref}
\LastPageEnding


\begin{thebibliography}{99}
\footnotesize\itemsep=-1pt

\bibitem{Adams2}
Adams C.R., On the irregular cases of the linear ordinary difference equation,
 \href{https://doi.org/10.2307/1989081}{\textit{Trans. Amer. Math. Soc.}} \textbf{30} (1928), 507--541.

\bibitem{Adams1931}
Adams C.R., Linear {$q$}-difference equations, \href{https://doi.org/10.1090/S0002-9904-1931-05162-4}{\textit{Bull. Amer. Math. Soc.}}
 \textbf{37} (1931), 361--400.

\bibitem{CdGP}
Candelas P., de~la Ossa X.C., Green P.S., Parkes L., A pair of {C}alabi--{Y}au
 manifolds as an exactly soluble superconformal theory, \href{https://doi.org/10.1016/0550-3213(91)90292-6}{\textit{Nuclear
 Phys.~B}} \textbf{359} (1991), 21--74.

\bibitem{CR13}
Chiodo A., Ruan Y., Landau--{G}inzburg/{C}alabi--{Y}au correspondence for
 quintic three-folds via symplectic transformations, \href{https://doi.org/10.1007/s00222-010-0260-0}{\textit{Invent. Math.}}
 \textbf{182} (2010), 117--165, \href{https://arxiv.org/abs/0812.4660}{arXiv:0812.4660}.

\bibitem{GS21}
Garoufalidis S., Scheidegger E., On the quantum {K}-theory of the quintic,
 \href{https://doi.org/10.3842/SIGMA.2022.021}{\textit{SIGMA}} \textbf{18} (2022), 021, 20~pages, \href{https://arxiv.org/abs/2101.07490}{arXiv:2101.07490}.

\bibitem{GR04}
Gasper G., Rahman M., Basic hypergeometric series, 2nd ed., \textit{Encyclopedia of
 Mathematics and its Applications}, Vol.~96, \href{https://doi.org/10.1017/CBO9780511526251}{Cambridge University
 Press}, Cambridge, 2004.

\bibitem{givental1998}
Givental A., A mirror theorem for toric complete intersections, in Topological
 Field Theory, Primitive Forms and Related Topics ({K}yoto, 1996),
 \textit{Progr. Math.}, Vol.~160, \href{https://doi.org/10.1007/978-1-4612-0705-4_5}{Birkh\"auser Boston}, Boston, MA, 1998,
 141--175.

\bibitem{givental2000wdvv}
Givental A., On the {WDVV} equation in quantum {$K$}-theory, \href{https://doi.org/10.1307/mmj/1030132720}{\textit{Michigan
 Math.~J.}} \textbf{48} (2000), 295--304, \href{https://arxiv.org/abs/math.AG/0003158}{arXiv:math.AG/0003158}.

\bibitem{givental2015}
Givental A., Permutation-equivariant quantum $K$-theory V.~Toric
 $q$-hypergeometric functions, \href{https://arxiv.org/abs/1509.03903}{arXiv:1509.03903}.

\bibitem{GDZ21}
Gu W., Pei D., Zhang M., On phases of 3d {$\mathcal{N} = 2$}
 {C}hern--{S}imons-matter theories, \href{https://doi.org/10.1016/j.nuclphysb.2021.115604}{\textit{Nuclear Phys.~B}} \textbf{973}
 (2021), 115604, 20~pages, \href{https://arxiv.org/abs/2105.02247}{arXiv:2105.02247}.

\bibitem{lee2004quantum}
Lee Y.-P., Quantum {$K$}-theory. {I}.~{F}oundations, \href{https://doi.org/10.1215/S0012-7094-04-12131-1}{\textit{Duke Math.~J.}}
 \textbf{121} (2004), 389--424, \href{https://arxiv.org/abs/math.AG/0105014}{arXiv:math.AG/0105014}.

\bibitem{MR21}
Milanov T., Roquefeuil A., Confluence in quantum $K$-theory of weak Fano
 manifolds and $q$-oscillatory integrals for toric manifolds,
 \href{https://arxiv.org/abs/2108.08620}{arXiv:2108.08620}.

\bibitem{Ro19}
Roquefeuil A., Confluence of quantum $K$-theory to quantum cohomology for
 projective spaces, \href{https://arxiv.org/abs/1911.00254}{arXiv:1911.00254}.

\bibitem{RW21}
Ruan Y., Wen Y., Quantum $K$-theory and $q$-difference equations,
 \href{https://arxiv.org/abs/2109.02218}{arXiv:2109.02218}.

\bibitem{Sauloy00}
Sauloy J., Syst\`emes aux {$q$}-diff\'erences singuliers r\'eguliers:
 classification, matrice de connexion et monodromie, \href{https://doi.org/10.5802/aif.1784}{\textit{Ann. Inst.
 Fourier (Grenoble)}} \textbf{50} (2000), 1021--1071.

\bibitem{hardouin:hal-01959879}
Sauloy J., Analytic study of {$q$}-difference equations, in Galois Theories of
 Linear Difference Equations: an Introduction, \textit{Math. Surveys Monogr.},
 Vol.~211, Amer. Math. Soc., Providence, RI, 2016, 103--171.

\end{thebibliography}
\end{document}